\documentclass[a4paper,11pt]{amsart}

\usepackage{amsfonts}
\usepackage{amssymb}
\usepackage{mathrsfs}
\usepackage{graphics}
\usepackage{tikz}
\usetikzlibrary{matrix,arrows}

\newcommand{\D}{\mathbb{D}}
\newcommand{\Cpp}{\mathbb{CP}^1}

\newcommand{\C}{\mathbb{C}}
\newcommand{\N}{\mathbb{N}}
\newcommand{\Z}{\mathbb{Z}}

\newcommand{\fol}{\mathcal{F}}
\newcommand{\tilf}{\widetilde{\mathcal{F}}}

\newcommand{\calp}{{\mathcal{P}}}

\newcommand{\R}{\mathbb{R}}

\newlength{\dhatheight}

\newcommand{\CP}{{\mathbb{C} \mathbb{P}}}

\usepackage[normalem]{ulem}
\usepackage{color}

\def\picill#1by#2(#3)#4
{\vbox to #2
	{\hrule width #1 height 0pt depth 0pt
		\vfill\special{illustration #3 scaled #4}}}

\newtheorem{theorem}{Theorem}[section]
\newtheorem{prop}[theorem]{Proposition}

\newtheorem{lemma}[theorem]{Lemma}

\theoremstyle{definition}
\newtheorem{defi}[theorem]{Definition}

\theoremstyle{remark}
\newtheorem{remark}[theorem]{Remark}

\usepackage{hyperref}

\begin{document}

\title[Integrability and meromorphic solutions]{Integrability of vector fields and meromorphic solutions}

\author[J.C. Rebelo]{Julio C. Rebelo}
\address{Institut de Math\'ematiques de Toulouse ; UMR 5219, Universit\'e de Toulouse, 118 Route de Narbonne, F-31062 Toulouse, France.}
\email{\href{mailto:rebelo@math.univ-toulouse.fr}{rebelo@math.univ-toulouse.fr}}

\author[H. Reis]{Helena Reis}
\address{Centro de Matem\'atica da Universidade do Porto, Faculdade de Economia da Universidade do Porto, Portugal.}
\email{\href{mailto:hreis@fep.up.pt}{hreis@fep.up.pt}}


\subjclass[2010]{Primary 34M05, 37F75; Secondary 34A05.}
\keywords{Meromorphic solutions; Liouvillian first integral; Foliated Poincar\'e metric; Riccati and Turbulent foliations}

\begin{abstract}
Let $\fol$ be a one-dimensional holomorphic foliation defined on a complex projective manifold and consider a meromorphic vector field $X$ tangent to $\fol$.
In this paper, we prove that if the set of integral curves of $X$ that are given by meromorphic maps defined on $\C$
is ``large enough'',
then the restriction of $\fol$ to any invariant complex $2$-dimensional analytic
set admits a first integral of Liouvillean type. In particular, on $\C^3$, every rational vector field whose solutions are
meromorphic functions defined on $\C$ admits a non-empty invariant analytic set of dimension~$2$ where the restriction of the
vector field yields a Liouville integrable foliation.
\end{abstract}

\dedicatory{To our friend Yulij S. Il'yashenko, wishing him many happy returns on the occasion of his ${\rm 80}^{\rm th}$ birthday}

\maketitle
\setcounter{tocdepth}{1}

\section{Introduction}

This paper is a contribution to a topic going back to the works of Painlev\'e, Chazy, Garnier et al. which is also closely
connected with distribution value theory in the spirit of Nevanlinna theory (see for example the survey \cite{whoever}). More
precisely, the object of this paper are systems of complex
differential equations whose solutions are uniform functions defined on $\C$ having only ``pole singularities''.
Systems with these characteristics appear in a variety of contexts, especially in Mathematical-Physics, and
they are also distinguished 
by exhibiting remarkable behaviors often shedding light in apparently unrelated problems.

Throughout the paper, we will work in the global setting of (holomorphic/meromorphic) {\it vector fields}\, defined on (compact) projective manifolds.
Clearly, this setting encompasses the case of differential equations, or systems of differential equations, appearing in the classical literature.
As to terminology, given a vector field
$X$, the phrases {\it solution of $X$}\, and {\it integral curve of $X$}\, will be used as synonyms.
We also note that our definition of meromorphic solution of $X$ basically means that the solution
in question is defined on $\C$
except by a discrete set where it is allowed to have poles, see Section~\ref{Preliminaries} for detailed definitions. Beyond classical
meromorphic functions, our setting also includes systems that can be integrated by means of
elliptic functions and of certain Painlev\'e transcendents (see below for further examples).

In the sequel, we will first state our main results, namely Theorems~A and~B, and then proceed to a more detailed discussion
of their applications along with examples of situations covered by them.

The definition of Liouvillean integrability can be found in Section~\ref{proofsoftheorems} or, alternatively,
in any standard reference such as \cite{singer}. We note, however, that the corresponding Liouville integrable foliations
will explicitly be classified in the course of the discussion (see Theorem~\ref{classification-foliation} as well as the proof of Theorem~A).
In other words, the notion of Liouvillean integrability is only used here as a unifying property allowing us to have succinct statements
for Theorems~A and~B, as opposed to a list of explicit possibilities.

\vspace{0.2cm}

\noindent {\bf Theorem A}. {\sl Let $X$ be a meromorphic vector field on an $n$-dimensional projective manifold $M$ and denote by $\fol$
the associated singular holomorphic foliation.
Denote by $\mathcal{M} \subset M$ the set of leaves of $\fol$ parameterized by integral curves
$\phi (T)$ of $X$ defining a meromorphic map from $\C$ to $M$ (and note that two integral curves parameterizing the same leaf differ
by a translation in time).
Assume that the Hausdorff dimension of $\mathcal{M} \subset M$ is strictly greater than $2n-2$. Then:
\begin{itemize}
	\item No enlarged leaf of $\fol$ (see Section~\ref{Preliminaries}) is a hyperbolic Riemann surface.
	
	\item For every invariant irreducible $2$-dimensional analytic set $N$ not contained in ${\rm Sing}\, (\fol)$, the
	restriction of $\fol$ to $N$ admits a non-constant first integral of Liouvillean type.
\end{itemize}
}

\vspace{0.2cm}

Before proceeding further, some comments about the above statement are in order. First, loosely speaking, the fundamental assumption
of Theorem~A is that the set $\mathcal{M} \subset M$ is ``large'' in a suitable sense. This can be made more accurate: basically, all it is needed
is to ensure that $\mathcal{M}$ is too large to be the polar set of a non-trivial plurisubharmonic function. The assumption
that the Hausdorff dimension of $\mathcal{M}$ exceeds $2n-2$ ensures this is the case, see Theorem~3.13 in
\cite{Who??}. Variants of this condition
can also be used, for example, we might as well assume that the logarithmic capacity of $\mathcal{M}$ is strictly positive in the spirit
of Suzuki's classical works, \cite{Suzuki}. In most cases, however, none of these conditions are easier to verify than the simpler
assumption that the Lebesgue measure of $\mathcal{M}$ is strictly positive.

As a second remark, we might also notice that
Theorem~A is sharp in the sense that meromorphic first integrals do not exist in general. For example, the
{\it Airy vector field} defined by
$$
\frac{\partial}{\partial z_1}  + z_3 \frac{\partial}{\partial z_2} + z_1 z_2 \frac{\partial}{\partial z_3}
$$
has holomorphic solutions on $\C^3$ but the underlying foliation can be compactified so as to have
Zariski-dense leaves on suitable invariant $2$-dimensional analytic sets (cf. \cite{formal}).

Next, note that the assumption of Theorem~A is birationally invariant so that the statement can be applied in every birational model of
$M$, $\fol$, or $X$. In particular, it becomes particularly effective in dimension~$3$ basically due to two reasons, namely:
\begin{itemize}
	\item The divisor of poles of $X$ has dimension~$2$ and is necessarily invariant under $\fol$ provided that all solutions
	of $X$ are meromorphic functions.
	
	\item In dimension~$3$, a result due to McQuillan and Panazzolo allows one to reduce the singularities of the foliation $\fol$
	by means of birational transformations, see \cite{daniel2}, \cite{danielMcquillan} (see also \cite{helenajulio_RMS} for a different
	proof).
\end{itemize}

Now, we can state Theorem~B as follows:

\vspace{0.2cm}

\noindent {\bf Theorem B}. {\sl Let $X$ be a meromorphic vector field on a projective manifold $M$ of dimension~$3$
	and denote by $\fol$ the associated singular holomorphic foliation. Assume that the solutions of $X$ are meromorphic functions
	defined on $\C$. Then, there exists a birational model $M'$, $\fol'$, $X'$ for $M$, $\fol$, and $X$ where one of
	the following holds:
	\begin{itemize}
		\item[(1)] $X'$ is a holomorphic vector field on $M'$ and all singular points of $\fol'$ are elementary,
		i.e., they have at least one eigenvalue different from zero.

		\item[(2)] $X'$ has non-empty divisor of poles $(X')_{\infty}$ which is fully invariant by $\fol'$ and such that the restriction of $\fol'$
		to any irreducible component of $(X)_{\infty}$ has a Liouvillean first integral. Moreover, all singular points of $\fol'$
		are elementary.
	\end{itemize}}

\vspace{0.1cm}

\noindent {\bf Remark}. In item~(2), the resulting manifold $M'$ may contain finitely many orbifold singular points of type $\Z /2\Z$. Naturally,
at these points, the notion of eigenvalues must be adequately defined, see \cite{daniel2}, \cite{danielMcquillan}, or \cite{helenajulio_RMS}.

Note that in the case $M = \CP^3$, Theorem~B already brings some progress to the classical problem of understanding
systems of rational differential equations on $\C^3$ with meromorphic solutions defined on $\C$.

Also, when
item~(1) above holds, there are several techniques to investigate in detail the pair $M$ and $X$, see
for example \cite{Akh} and \cite{EMS}. Meanwhile, when the situation described in item~(2) happens,
Theorem~B provides a setting very much comparable to the analogous
$2$-dimensional situation exploited in \cite{Guillot-Rebelo}. In fact, Theorem~B adequately deals with one of the fundamental difficulties one
faces when trying to extend to dimension~$3$
some important results known in dimension~$2$ such as those found in \cite{brunellacomplete}, 
\cite{Guillot-Rebelo}, \cite{bustinduy}, \cite{alvaro-Giraldo-1}, \cite{alvaro-Giraldo-2}, \cite{guillotadvances} to name only a few.
This statement deserves further elaboration.
For this, it is enough to restrict our discussion to the already very interesting case of vector fields that are either
strictly rational or polynomial with degree at least~$2$ and that have
meromorphic solutions defined on $\C$.
In particular, this includes the case of {\it complete}\, polynomial vector fields on $\C^3$. Here, the reader is reminded
that a holomorphic vector field is said to be {\it complete}\,
if its solutions are holomorphic functions defined on all of $\C$ so that these vector fields naturally fulfill the conditions of Theorems~A or~B.
The setting of Theorem~B also covers the examples provided by
Painlev\'e equations P-I, P-II, P-IV, along with the modified versions of P-III and P-V, 
as well as the preliminary results obtained by Garnier in the case of ``special systems of equations'',
see \cite{Ince}, \cite{gromak}, \cite{garnier}.
Another very interesting example of polynomial vector field on $\C^3$ satisfying the
conditions of Theorem~B and related to both Painlev\'e and Poincar\'e problems can be found in \cite{guillotCRAS}, \cite{alcides}.
Plenty of additional examples can be found in \cite{guillotChazy}.
In a similar order of ideas, Nagata's celebrated automorphism of $\C^3$ is time-one map of a complete
polynomial vector field \cite{nagata}.

%
%
%

To appreciate the contribution brought by Theorem~B to the area, it suffices to consider the case of {\it homogeneous}
complete polynomial vector fields on $\C^3$. If $X$ is one such vector field,
then its singular point at the origin --- or equivalently its behavior on
a neighborhood of the plane at infinity $\Delta$ --- can be investigated by means of blow ups. This quickly leads
to a number of constraints in the structure of these vector fields as shown in \cite{guillotFourier} in the more general context
of semicomplete vector fields (a vector field is said to be semicomplete if all of its solutions admit a maximal domain
of definition in $\C$, see Remark~\ref{defnc_semicomplete-vf} for detail).
However, even in certain favorable cases, we are not
able to decide whether or not these {\it necessary conditions}\, are also sufficient. More precisely, in general, the vector field $X$
induces a foliation on $\Delta$ whose dynamics may --- in principle --- be very complicated (cf. \cite{lorayandjulio}) and this
prevents us from fully understanding the structure of $X$ so as to decide whether or not the vector field
is semicomplete. It is exactly this difficulty that is lifted by Theorem~B when the vector fields are {\it complete} since
it is straightforward to derive from Theorem~B and from Theorem~\ref{classification-foliation} a classification of homogeneous
complete polynomial vector fields on $\C^3$.

In particular, in the wide open problem of classifying (non-homogeneous) complete polynomial vector fields on $\C^3$,
Theorem~B still plays the
same important role. In fact, following the point of view of \cite{Guillot-Rebelo},
Theorem~B basically provides a pretty accurate description of the structure of $X$ around $\Delta$. A systematic investigation of the possible
structures, in turn, is bound to bring significant
progress on the problem of classifying the vector fields in question. As a side note, we might also refer to the survey \cite{survey_RR}
for a comprehensive discussion of ``core dynamics'' of singular points and their applications which is very much
the same issue tackled by Theorem~B in the case of vector fields with meromorphic solutions.

\noindent {\bf Outline of the paper}.
Let us close this introduction by outlining the structure of this paper. Throughout this work, a Riemann surface is said to
be {\it parabolic} (resp. {\it hyperbolic}) if it is a quotient of $\C$ (resp. the unit disc $\D$). A crucial point in the proof
of the preceding theorems is to be able to conclude that
{\it no leaf}\, of a foliation $\fol$ is hyperbolic provided that there are ``sufficiently many'' non-hyperbolic leaves. Once
it is established that {\it all leaves}\, of $\fol$ are either parabolic or rational curves,
there will also follow that the same holds for every
birational model of the foliation in question: and this includes leaves contained in ``exceptional'' (small)
sets. With this conclusion in place, we will be able
to take advantage of the $2$-dimensional setting to derive Theorems~A and~B (more on this below).

The problem of ruling out the existence of hyperbolic leaves is, however, a subtle one. In particular, it has some
intersection with the general topic of {\it simultaneous parabolic uniformization}\, for foliations as discussed for example in
\cite{alexey}, \cite{etienne}. Compared to the quoted works, the fact that our foliations are singular poses some new challenges
as pointed out below.

Let $M$, $X$, and $\fol$ be as in Theorem~A. The leaves of $\fol$ are defined as the leaves of the regular foliation induced
by $\fol$ in the complement of its own singular set. In this sense, the fact that an integral curve of $X$ through a point
of $L$ yields a meromorphic map defined on all of $\C$ {\it does not}\, prevent $L$ from being a quotient of $\D$ due to
the possible presence of poles that might correspond to singular points of $\fol$. In fact, what this meromorphic map
yields is an {\it entire curve tangent to $\fol$}\, namely, a non-constant holomorphic map $\varphi : \C \rightarrow M$
whose image is invariant under $\fol$. An entire curve tangent to $\fol$ can therefore {\it contain singular points of $\fol$}\, which
inevitably raise the issue of deciding whether or not these singular points should be {\it considered as ``belonging'' to the
corresponding leaf of $\fol$}.
The possible inclusion of singular points of $\fol$ in leaves clearly
impacts the hyperbolic/hon-hyperbolic nature of the leaf in question. Furthermore, naturally, 
the possible (local) wild behavior of $\fol$ near its singular points adds to the general difficulties.
The following example,
borrowed from \cite{guillot-IHES}, further illustrates the subtle nature of these issues. In his study of
(hyperbolic) Halphen vector fields \cite{guillot-IHES}, A. Guillot has
constructed examples of {\it complete}\, holomorphic vector fields $X$ on suitable complex manifolds $M$ of dimension~$3$
possessing an isolated singular point $P$ where $X$ is conjugate to a (quadratic, hyperbolic) Halphen vector field.
In this construction, all regular leaves of the underlying foliation $\fol$ are parabolic since $X$ is complete with isolated singular points.
However, by blowing up the singular point $P$, we obtain a new vector field
$\widetilde{X}$, still complete, on a manifold $\widetilde{M}$. The vector field $\widetilde{X}$ has a non-empty divisor of zeros
which coincides with the exceptional divisor $\Pi^{-1} (P)$ (isomorphic to $\CP^2$). Furthermore, the exceptional divisor $\Pi^{-1} (P)$
is invariant by the transformed foliation
$\widetilde{\fol}$. In particular, there are leaves of $\widetilde{\fol}$ that are contained in $\Pi^{-1} (P)$. It turns out, however, as
shown in \cite{guillot-IHES}, that except for three rational curves (minus singular points), all
the leaves of $\widetilde{\fol}$ contained in $\Pi^{-1} (P)$ are --- genuinely --- hyperbolic Riemann surfaces.
This contrasts with the fact that regular leaves of $\widetilde{\fol}$ not contained in $\Pi^{-1} (P)$
are quotients of $\C$. The co-existence of ``many'' parabolic leaves and hyperbolic leaves is striking from the perspective of Theorem~A.
In addition, the restriction of
$\widetilde{\fol}$ to $\Pi^{-1} (P)$ has a complicated dynamics far beyond the realm of Liouville integrable systems.
The apparent contradiction between this example and our results is explained by the fact that the manifolds
where hyperbolic Halphen vector fields become complete are far from having a K\"ahler structure.

Despite what precedes,
in the algebraic setting (or on K\"ahler compact manifolds), there is a general theorem due to M. Brunella
about simultaneous uniformization: the foliated Poincar\'e metric varies in a plurisubharmonic way, up to
an adequate definition of ``leaf'' \cite{subharmonic-variation}, \cite{Brunella-notes}, \cite{Ivaskovich}.
We will make fundamental use of this theorem to conclude that no ``leaf'' of a foliation
as in Theorem~A can be a quotient of $\mathbb{D}$ provided that the definition of ``leaf'' is suitably modified.
In fact, the specific notion of ``leaf'' required for Brunella's theorem to hold often makes
applications of his theorem less immediate. In the present case, the issue is
further compounded by the fact that we made no assumption on how the poles of the solutions of $X$ vary with the initial condition.
Having no control on the behavior of poles prevents us
from resorting to standard techniques about extension of analytic sets \cite{bishop},
\cite{chirka} that would be helpful if, for example, poles were assumed to vary analytically with the initial condition
(cf. Section~\ref{Appendix}).
Basically, the key to overcome these issues turns out to be a remark about singularities of $1$-dimensional foliations
with ``many separatrices'', along with their behavior under sequences of blow ups (see Section~\ref{Dicriticalnessoffoliations}).
In fact, a by-product of the discussion conducted in Section~\ref{Dicriticalnessoffoliations} is a different interpretation
of Brunella's definition of ``leaf'' used in the above quoted papers. As a consequence, in an appropriate sense,
it can be said that from the point of view of potential theory,
Brunella's theorem in \cite{subharmonic-variation} and related issues are all about {\it regular foliations}.
We believe this remark holds interest in itself and might be helpful in general problems about regularity of the foliated
hyperbolic metric. Furthermore, results on the variation of the polar set of solutions with the initial conditions can
straightforwardly be derived from the same discussion.

Once we know that no leaf of $\fol$ is hyperbolic, there follows that the leaves of the restriction of $\fol$ to an invariant (singular)
analytic surface are either parabolic or rational curves. Much information on the structure of these foliations in $2$-dimensional
ambient is available in Brunella's notes \cite{IMPA-book}. Up to adding some  standard material
involving (global, holomorphic) vector fields on algebraic surfaces and some basic facts about Kleinian groups Theorem~A
will follow. By then, Theorem~B will be an easy consequence of Theorem~A and the general results about reduction of singularities
for $1$-foliations presented in \cite{daniel2}, \cite{danielMcquillan}, and \cite{helenajulio_RMS}.

\vspace{0.2cm}

\noindent
\textbf{Background on Hausdorff dimension}. The reader less familiar with the notion of Hausdorff dimension emphasized in this
paper will find the necessary background information in \cite{Who??} (Theorem~3.13) and in, for example, \cite{falconer}. The results
used in the course of this paper can be summarized by the following statements:
\begin{itemize}
	\item A countable union of sets with Hausdorff dimension strictly smaller than $\kappa$ still has Hausdorff dimension strictly smaller than~$\kappa$.
	
	\item If $\mathcal{M}$ is a set where the $d$-dimensional Lebesgue measure is well defined, i.e., it yields a non-zero locally finite
	measure on $\mathcal{M}$, then the Hausdorff dimension of $\mathcal{M}$ equals~$d$.
	
	\item If $M$ is a projective (or more generally K\"ahler) compact manifold of complex dimension~$n$, then no subset $\mathcal{M}$
	whose Hausdorff dimension is strictly greater than $2n-2$ can be contained in a (local) polar set of a non-trivial plurisubharmonic function.
\end{itemize}
In particular, as mentioned before, the role of Hausdorff dimension in the statement of Theorem~A can be adapted to various types
of ``capacities'' used in potential theory.

\vspace{0.2cm}

\noindent
\textbf{Acknowledgments}. J. Rebelo and H. Reis were partially supported by CIMI through the
project ``Complex dynamics of group actions, Halphen and Painlev\'e systems''.
H. Reis was also partially supported by CMUP, which is financed by national funds through FCT – Funda\c{c}\~ao para
a Ci\^encia e Tecnologia, I.P., under the projects UIDB/00144/2020 and ``Means and Extremes in Dynamical Systems''
with reference PTDC/MAT-PUR/4048/2021.

This paper was motivated by discussions with A. Belotto and M. Klimes, we are grateful to both of them for their comments.
Last but not least, we thank the referee for carefully reading our work and for his/her very valuable suggestions.


\section{Vector fields, foliations, and enlarged leaves}\label{Preliminaries}

It is convenient to begin by recalling the definition of singular $1$-dimensional holomorphic foliation on a complex manifold,
also known as (singular) holomorphic foliation by Riemann surfaces. Consider a covering of a complex manifold $M$ by coordinate charts
$\{ (U_k, \varphi_k)\}$.

\begin{defi}
	\label{foliation_definition}
	Let $M$ and $\{ (U_k, \varphi_k)\}$ be as above. A singular holomorphic foliation $\fol$ on $M$ consists of a collection
	of holomorphic vector fields $Y_k$ satisfying the following conditions:
	\begin{itemize}
		\item For every $k$, $Y_k$ is a holomorphic vector field defined on $\varphi_k (U_k) \subset \C^n$ whose
		singular set has codimension at least~$2$.
		
		\item Up to the evident identification of $Y_k$ with a vector on $U_k$, whenever $U_{k_1} \cap U_{k_2} \neq \emptyset$
		we have $Y_{k_1} = g_{k_1 k_2} Y_{k_2}$ for some $g_{k_1 k_2} \in \mathcal{O}^{\ast} (U_{k_1} \cap U_{k_2})$.
		
	\end{itemize}
\end{defi}

The {\it singular set}\, ${\rm Sing}\, (\fol)$ of a foliation $\fol$ is then defined as the union over~$k$ of the sets $\varphi_k^{-1}
({\rm Sing}\, (Y_k)) \subset M$, where ${\rm Sing}\, (Y_k)$ stands for the singular set of $Y_k$. Therefore
the singular set of any $1$-dimensional holomorphic foliation has codimension at least two.
To abridge notation, the phrase {\it holomorphic foliation}\, will mean
a singular $1$-dimensional holomorphic foliation.
For the basics of foliation theory as required for this work, the reader is referred to any of the standard books in the area.
For example, \cite{standardfoliation} or \cite{IY} are both well adapted to our needs.

Consider a complex manifold $M$ equipped with a non-trivial meromorphic vector field $X$ whose divisor of poles is denoted by
$(X)_{\infty}$. Away from $(X)_{\infty}$ and from the zero-set of $X$, the local orbits of $X$
define a regular holomorphic foliation. The interest of Definition~\ref{foliation_definition} largely
stems from the fact that this foliation can be holomorphically extended to all of $M$. Indeed, we have:

%
%
%
%
%

\begin{lemma}
\label{lemma_meromorphicvfieldfoliation}
Let $X$ be a meromorphic vector field defined on a complex manifold $M$. Then the local orbits of $X$ induce a
holomorphic foliation $\fol$ defined on all of $M$.
\end{lemma}

\begin{proof}
Consider a point $p \in M$. By definition, in a suitable coordinate $\varphi_k : U \subset M \rightarrow \C^n$
defined around $p$, the vector field $X$ is represented by a meromorphic vector field $X_{k}$ defined
on $\varphi_k (U) \subset \C^n$. Owing to the standard Hilbert nullstellensatz in codimension~$1$, the divisor of zeros and poles
of $X_{k}$ can be factored out so as to have
$$
X_k = \frac{f}{g} \, Y_k \, ,
$$
where $f$ and $g$ are holomorphic functions and where $Y_k$ is a holomorphic vector field whose singular set has codimension at least~$2$.
It is now immediate
to check that the collection of vector fields $\{ Y_k\}$ endows $M$ with a singular holomorphic foliation $\fol$ according to
Definition~\ref{foliation_definition}.
\end{proof}

The foliation $\fol$ arising from a vector field $X$ as in Lemma~\ref{lemma_meromorphicvfieldfoliation} will be called the
{\it foliation associated with $X$} or the {\it underlying foliation of $X$}. In this paper, the 
{\it leaves}\, of a singular holomorphic foliation $\fol$ on a manifold $M$ are defined as the leaves of the (regular) foliation
obtained by restricting $\fol$ to the open manifold $M \setminus {\rm Sing}\, (\fol)$.

\begin{remark}
{\rm It is worth emphasizing the following consequence of Lemma~\ref{lemma_meromorphicvfieldfoliation}.
Consider a meromorphic vector field $X$ defined on $M$ and assume that the divisor of zeros and poles
of $X$ is not empty. Then no irreducible component of this divisor can be contained in the singular set of $\fol$ since the latter
set has codimension at least~$2$. Hence the ``generic'' point of the divisor
of zeros and poles of $X$ is regular for $\fol$. Furthermore, irreducible components of this divisor may or may not be invariant
by $\fol$. In particular, an invariant component will contain regular leaves of the foliation $\fol$ despite the fact that $X$
either vansihes identically or is not even defined on the component in question.}
\end{remark}

Next, let us make accurate the notion of {\it meromorphic solution defined on $\C$}\, for a vector field.
For the general definition of meromorphic map the reader is referred to \cite{remmert} or to Section~\ref{meromorphicsolutions_brunellatheorem}.
Consider a meromorphic vector field $X$ on a complex manifold $M$ and denote by $\fol$ its associated foliation.

\begin{defi}\label{defining_meromorphic-integralcurve}
Given a point $p \in M$ that is regular for $X$, we
say that {\it the solution of $X$ through $p$ is a meromorphic map (function) defined
on $\C$}\, if the local solution $\phi$ of $X$ satisfying $\phi (0) = p$ fulfills the following conditions:
\begin{itemize}
	\item[(a)] There is a discrete set $\mathcal{D} = \{ t_i\}_{i \in \N} \subset \C$ such that $\phi$ possesses a holomorphic extension $\Phi$ to
	$\C \setminus  \mathcal{D}$.
	
	\item[(b)] The points $t_i \in \mathcal{D}$ are poles for $\Phi$, i.e., $\Phi : \C \rightarrow M$ is a meromorphic map.
\end{itemize}
\end{defi}
Clearly, if the solution of $X$ through $p$ is a meromorphic map on $\C$, then the same applies to every point in the orbit
of $p$. We will then say that {\it the leaf $L$ of $\fol$ through $p$ is meromorphically
parameterized by $\C$}. Up to dropping the condition of $\phi$ being a solution of $X$, we note that
the preceding definition makes sense for every leaf of $\fol$, including those contained in the divisor of zeros and poles of $X$.
In fact, if we will say that a leaf $L$ of $\fol$ is {\it meromorphically parameterized}\, by $\C$ if there exists some (non-constant) meromorphic
map $\phi : \C \rightarrow M$ satisfying items~(a) and~(b) above and such that $\phi (\C \setminus  \mathcal{D}) = L$. Finally,
we will say that one such parameterization $\phi$ is obtained {\it via $X$}\, if $\phi$ is a solution of $X$.

\begin{remark}
{\rm From the above definition we see that vector fields that can be integrated by means of elliptic functions fit in the framework
of Theorems~A and~B. The same applies to the vector fields arising from Painlev\'e~I, II, and~IV equations, see \cite{gromak}.
Several other examples were already quoted in the introduction.}
\end{remark}

Now, we have:

\begin{lemma}
\label{lemma_invaraincepoledivisor}
Let $X$ be a meromorphic vector field defined on a complex manifold $M$ and denote by $S$ an irreducible
(non-empty) component of the divisor of poles of $X$. If $p \in S$ is a regular point of $\fol$ at which
$\fol$ and $S$ are transverse, then the leaf $L$ of $\fol$ through $p$ is not meromorphically
parameterized by $\C$ via $X$.
\end{lemma}

\begin{proof}
Consider the restriction $X_{\vert L}$ of $X$ to $L$ viewed as a meromorphic vector field on the Riemann surface $L$. Let $z$
be a local coordinate around $p\in L$, $p \simeq 0 \in \C$. Since $X$ has a pole at $p$, there follows that $X$ takes on
the form
$$
X = z^{-k} f(z) \partial /\partial z
$$
in the coordinate $z$. Here $k \geq 1$ and $f$ is a holomorphic function satisfying $f(0) \neq 0$. In fact, the elementary theory
of normal forms for $1$-dimensional complex vector fields ensures that the coordinate $z$ can be chosen so as to have
$f=1$ (constant). In other words, without loss of generality we can suppose that $X =  z^{-k} \partial /\partial z$. The integral curve
of $X$ with initial point $z_0 \neq 0$ is therefore $\phi (t)  = \sqrt[k+1]{(1+k)t + z_0^{k+1}}$. This function is multivalued
since $k+1 \geq 2$ and therefore does not yield a meromorphic function defined on $\C$. This establishes the lemma.
\end{proof}

\begin{remark}\label{defnc_semicomplete-vf}
{\rm The statement of Lemma~\ref{lemma_invaraincepoledivisor} actually holds for semicomplete vector fields, see
for example \cite{Guillot-Rebelo}. Note that vector fields all of whose solutions are meromorphic in the sense of
Definition~\ref{defining_meromorphic-integralcurve} are necessarily semicomplete. This statement follows immediately
from the definition of semicomplete vector fields which we recall for the convenience of the reader. Let then $X$
be a holomorphic vector field defined on a complex
manifold $U$. We say that $X$ is {\it semicomplete on $U$}\, if for every point $p \in U$, there exists a solution $\phi : \Omega_p \subset
\C \rightarrow U$, with $\phi (0) =p$, which is {\it maximal} in the following sense: for every sequence $\{ T_i \} \subset \Omega_p$
of points converging to a point $\hat{T}$ lying in the boundary of $\Omega_p$, the sequence $\{ \phi (T_i) \} \subset U$ leaves every
compact set contained in $U$. Slightly more generally, a meromorphic vector field is said to be semicomplete on $U$ if its
semicomplete away from its pole divisor in $U$. Naturally this condition is always satisfied by meromorphic vector fields all of whose
solutions are meromorphic so that the claim follows.}
\end{remark}

Brunella's theorem on the regularity of the leafwise Poincar\'e metric requires a more specific notion of ``leaf'' allowing,
in some cases, a singular point of $\fol$ to ``belong to a leaf'', see \cite{subharmonic-variation} or \cite{Brunella-notes}.
To avoid misunderstandings, throughout this work, the word {\it leaf}
(of $\fol$) will always be meant in the sense so far used: a leaf of the regular foliation induced by $\fol$ on $M \setminus {\rm Sing}\, (\fol)$.
In turn, the variant of the notion of leaf required
by Brunella's theorem will be referred to as {\it enlarged leaf}. As indicated,
an {\it enlarged leaf}\, $\widehat{L}$ of $\fol$ is obtained out of a leaf $L$ of $\fol$ by (possibly) adding singular
points of $\fol$ to $L$. The rule to determine whether or not a singular point of $\fol$ should be ``added'' to a given leaf involves only
local information so that we can place ourselves
in the case of a foliation $\fol$ defined on the $n$-dimensional polydisc $\D^n \subset \C^n$ around the origin
(cf. \cite{Brunella-notes}).

On the polydisc $\D^n$, we consider the
trivial fibration $\D^n = \D^{n-1} \times \D \rightarrow \D^{n-1}$. A meromorphic map $f :
\D^n \rightarrow M$ is said to be a {\it foliated meromorphic immersion}\, if the indeterminacy set $I (f)$ of $f$
intersects each vertical fiber of $\D^n$ in a discrete set and if $f$ satisfies the following additional conditions:
\begin{itemize}
\item $f$ is an immersion on the complement of $I (f)$.
	
\item In the complement of $I (f)$, $f$ takes vertical fibers to leaves of $\fol$.
\end{itemize}

Next, if $L$ is a leaf of $\fol$, a closed subset $K \subset L$ is said to be a {\it vanishing end}\, of $L$ if the
following conditions are satisfied:
\begin{itemize}
	
\item[(1)] $K$ is isomorphic to the punctured disc and the holonomy of the restriction of $\fol$ to $M \setminus
	{\rm Sing}\, (\fol)$ corresponding to the cycle $\partial K$ has finite order $k$.
	
\item[(2)] There is a foliated meromorphic immersion $f : \D^n \rightarrow M$ such that
	\begin{itemize}
		\item[(2a)] $I (f) \cap (\{ 0\} \times \D)$ is reduced to the origin of $\C^n$ (and where ``$\{ 0 \}$'' stands for the
		origin of $\D^{n-1} \subset \C^{n-1}$).
		
		\item[(2b)] The image of $f$ restricted to $\{ 0\} \times \D$ coincides with the interior of $K$.  Furthermore $f: \{ 0\}
		\times \D \rightarrow {\rm Int} \, (K)$ is a regular covering of degree $k$, where ${\rm Int} \, (K)$ stands for the
		interior of $K$.
	\end{itemize}
\end{itemize}

Let now $p$ be a point in $M \setminus {\rm Sing}\, (\fol)$ and denote by $L_p$ the leaf of $\fol$ through $p$.
Next, add to $L_p$ all of its vanishing ends
where the operation of adding an end to $L_p$ should be
understood in the sense of orbifolds: the multiplicity of the added point will be the order $k$ of the holonomy
relative to $\partial K$. This construction gives rise to a Riemann surface orbifold which, in turn, can be turned into
a Riemann surface by standard normalization.

\begin{defi}\label{defining_enlargedleaf}
The enlarged leaf $\widehat{L}_p$ of $\fol$ through $p$ is the Riemann surface obtained from $L_p$ after normalizing the
above constructed orbifold.
\end{defi}

Naturally there is little difference whether we consider the enlarged leaf $\widehat{L}_p$ as an orbifold or as a (normal)
Riemann surface. For reference, it is convenient to point out the following immediate consequence of Definition~\ref{defining_enlargedleaf}.

\begin{lemma}
\label{vanishingends_discreteset}
Given $p \in M \setminus {\rm Sing}\, (\fol)$, let $L_p$ and $\widehat{L}_p$ denote, respectively, the leaf and the enlarged
leaf of $\fol$ through $p$. Then the set $\widehat{L}_p \setminus L_p$ is a {\it discrete set of $\widehat{L}_p$}.\qed
\end{lemma}

The content of Lemma~\ref{blowup_invariance} below is another straightforward consequence of the preceding construction.
Let $M$, $\fol$, and  ${\rm Sing}\, (\fol)$ be as above. Next, let $C \subset M$ be a smooth submanifold contained in the singular
set ${\rm Sing}\, (\fol)$ of $\fol$ (in particular $C$ has codimension at least~$2$).
Consider the blow up $\pi : \widetilde{M} \rightarrow M$ of $M$ centered at $C \subset {\rm Sing}\, (\fol) \subset M$.
The {\it blow up} of $\fol$, i.e., the transform of $\fol$ under $\pi$, will be denoted by
$\tilf = \pi^{\ast} \fol$.

\begin{lemma}
\label{blowup_invariance}
With the preceding notation, let $\widetilde{L}$ be a local branch of a leaf of $\tilf$ not contained in the exceptional divisor $\pi^{-1} (C)$
(so that $L = \pi (\widetilde{L})$ is a local branch of a leaf of $\fol$).
Assume that $\widetilde{L}$ possesses a vanishing end $\widetilde{K}$ and denote by $\widetilde{P}$
the corresponding end of $\widetilde{L}$ (so that $\widetilde{P}$ belongs to the enlarged leaf of $\tilf$ obtained from $\widetilde{L}$).
If $\widetilde{P} \in \pi^{-1} (C)$ then the singular point $P = \pi (\widetilde{P})$ of $\fol$ belongs to the enlarged
leaf of $\fol$ obtained from $L$.
\end{lemma}

\begin{proof}
Clearly $K = \pi (\widetilde{K})$ satisfies Condition~(1) in the definition of vanishing end
since $\pi$ is a diffeomorphism away from $\pi^{-1} (C)$. As to Condition~(2), if $\widetilde{f} : \D^n \rightarrow \widetilde{M}$
denotes the foliated meromorphic immersion associated to $\widetilde{K}$ then $f = \pi \circ \widetilde{f} : \D^n \rightarrow M$
provides the required meromorphic immersion for $K$ and $\fol$. The lemma follows.
\end{proof}

\begin{remark}
{\rm The converse of Lemma~\ref{blowup_invariance} also holds. In fact, note first that Condition~(1)
holds for $K$ if and only if it does for $\widetilde{K}$.
Next, given a foliated meromorphic immersion
$f : \D^n \rightarrow M$ corresponding to $K \subset L$, a foliated meromorphic immersion $F$ relative to
$\pi^{-1} (K)$ can be obtained by setting $F = \pi^{-1} \circ f$ on $\D^n \setminus f^{-1} (C)$ and observing that
$F$ extends meromorphically to all of $\D^n$ owing to Levi extension.}
\end{remark}

Let us close this section with some comments concerning Haefliger's point of view on the holonomy pseudogroup of a foliation
\cite{haefliger-II}, \cite{haefliger}
encoding all of the transverse dynamics of the foliation in question. This discussion will lead, in particular, to
Lemma~\ref{Leaveswithoutholonomy}. Albeit not fully indispensable, this lemma allows for
a more direct argument in the proof of Theorem~\ref{classification-foliation} (cf. Section~\ref{meromorphicsolutions_brunellatheorem}).
Another advantage of considering Haefliger's holonomy pseudogroup
is to make accurate the idea of ``transverse measures'' for foliations. These transverse measures and their basic properties,
mostly stemming from Fubini's theorem, are inevitable in the statements and in the proofs of our main results.
They will also help us to interpret Brunella's theorem in \cite{subharmonic-variation}
as a statement about regular foliations on suitable open manifolds, up to ignoring certain small sets that are
{\it negligible}\, from the point of view of classical potential theory.

Away from the singular set ${\rm Sing}\, (\fol)$ of $\fol$, we consider a {\it countable, locally finite}, covering $\{ U_i, \varphi_i ,
\Sigma_i \}_{i \in \N}$ of $M \setminus {\rm Sing}\, (\fol)$ by foliated coordinates. Here, we have $\varphi_i (U_i) = \D \times \Sigma_i
\subset \C^n$ equipped with coordinates $(x,y)$, $x \in \D \subset \C$ and $y \in \C^{n-1}$. In particular, each $\Sigma_i$ can be
identified with the unit ball in $\C^{n-1}$. Naturally, in these local coordinates, $\fol$ is represented by
the constant vector field $\partial /\partial x$. Next, if $U_i \cap U_j \neq \emptyset$, then the change of coordinates
$\varphi_j \circ \varphi_i^{-1} (x,y)$ takes on the form $(f_{ij} (x,y), \gamma_{ij} (y))$. Therefore, each local map
$\gamma_{ij}$ induces a diffeomorphism between open subsets of $\Sigma_i$ and $\Sigma_j$. The pseudogroup of maps between
open sets of the transverse space $\{ \Sigma_i\}$ generated by the collection $\{ \gamma_{ij} \}$ is called
the {\it Haefliger pseudogroup of $\fol$}, see \cite{haefliger-II}.

Now, by identifying $\Sigma_i$ with the unit ball ${\rm B}^{(n-1)}$ in $\C^{n-1}$, let each $\Sigma_i$ be equipped with the (finite) Lebesgue measure
$\mu_i$. In general, the resulting collection of spaces and measure $\{ (\Sigma_i, \mu_i)\}$ {\it is not} preserved by Haefliger
pseudogroup, i.e., in general we have $(\gamma_{ij})_{\ast} \mu_i \neq \mu_j$. However, the measures $(\gamma_{ij})_{\ast} \mu_i$
and $\mu_j$ clearly have the same null-measure sets. Thus, there is a sense to say that the Lebesgue measure of a set
invariant by $\fol$ is zero or strictly positive. Similarly, the Hausdorff dimension of sets is also invariant by all the local
diffeomorphisms $\gamma_{ij}$. In particular, if $A \subset {\rm B}^{(n-1)}$ has Hausdorff dimension strictly greater than
$2n-4$, then the same holds for all the sets $\gamma_{ij} (A) \subset {\rm B}^{(n-1)}$ (where the transverse sections
$\Sigma_i$ are identified with ${\rm B}^{(n-1)}$).

Let us close this section by proving a useful result (Lemma~\ref{Leaveswithoutholonomy}) on the set of leaves without holonomy that parallels a
well known theorem of Epstein-Millet-Tischler \cite{epsteinetal}. This statement also provides a standard application,
albeit in a slightly implicit form, of Haefliger's pseudogroup point of view.

\begin{lemma}
	\label{Leaveswithoutholonomy}
	Let $\fol$ be a singular holomorphic foliation on a complex K\"ahler manifold $M$. Then the set of leaves of $\fol$ carrying
	non-trivial holonomy has Hausdorff dimension at most $2n-2$. In particular, the Lebesgue measure of this set of leaves is zero.
\end{lemma}

\begin{proof}
	Basically, the proof amounts to applying some specific results from \cite{epsteinetal} to the present context in which the
	foliations are holomorphic. Let us provide a detailed account of the argument. Consider
	$\fol$ as a regular holomorphic foliation on the open complex manifold $M \setminus {\rm Sing}\, (\fol)$.
	Let $L$ be a leaf of $\fol$ and denote by $\sigma :[0,1] \rightarrow L$ a loop in $L$. Unless
	the holonomy map $\rho_L (\sigma)$ obtained from $\sigma$ is the identity, $L$ is contained in the set of fixed points
	of $\rho_L (\sigma)$ and the latter is a (local) proper complex analytic subset of $\bigcup_{i \in \N} \Sigma_i$, with the transverse sections
	$\Sigma_i$ being pairwise disjoint. In particular, the Hausdorff dimension of this set is at most $2n-4$ while its (transverse) Lebesgue
	measure is zero.

	Now, note that on any compact subset $\overline{V}$ of $M \setminus {\rm Sing}\, (\fol)$, the natural restriction of Haefliger pseudogroup
	is finitely generated. Consider then the restriction $\fol_{\vert \overline{V}}$ of $\fol$ to $\overline{V}$ along with a suitable auxiliary
	metric on $M$. Since $\fol_{\vert \overline{V}}$ is a regular foliation on $\overline{V}$, the general
	results of \cite{epsteinetal} can be brought to bear. In particular, fixed $l \in \R$, let $\mathcal{L}_l$ denote
	the set of leaves of $\fol_{\vert \overline{V}}$ possessing a non-trivial holonomy map arising from a loop of length less than $l$.
	According to \cite{epsteinetal}, there are finitely many transverse sets $F_s \subset \{ \Sigma_i\}$, with
	each $F_s$ contained in a transverse section $\Sigma_{i (s)}$, such that the following hold:
	\begin{itemize}
		\item[(A)] The saturated set by $\fol$ of the sets $F_s$ contains $\mathcal{L}_l$.

		\item[(B)] For every $F_s$, there is an open set $U \subset \Sigma_{i (s)}$ containing $F_s$ and some holonomy map
		$h : U \subset \Sigma_{i (s)} \rightarrow \Sigma_{i (s)}$ satisfying the two conditions below:
		\begin{itemize}
			\item[(B.1)] The map $h$ is a diffeomorphism from $U$ to its image that does not coincide with the identity on $U$.
			
			\item[(B.2)] $F_s$ is contained in the set ${\rm Fix}\, (h)$ of fixed points of $h$.
	    \end{itemize}	
	\end{itemize}
Now, taking into account that $\fol_{\vert \overline{V}}$ is a holomorphic foliation, there follows that $h$ is a holomorphic diffeomorphism
and thus that ${\rm Fix}\, (h)$ is a proper analytic set of $U \subset \Sigma_{i (s)}$. In particular, the Hausdorff dimension
of $F_s$ is at most $2n-4$ so that the Hausdorff dimension of $\mathcal{L}_l$ cannot exceed $2n-2$. We also point out that
the complement of $\mathcal{L}_l$ is open and dense though this is not needed here (c.f., the general $G_{\delta}$-dense sets produced
in \cite{epsteinetal}).

Summarizing what precedes, fixed $l$, the set $\mathcal{L}_l$ consisting of leaves of $\fol_{\vert \overline{V}}$ containing a loop of length
at most $l$ that carries non-trivial holonomy has Hausdorff dimension bounded by $2n-2$. Now, by considering the countable union
of these sets of leaves as $l$ runs over the positive integers, we conclude that the Hausdorff dimension
of the set $\mathcal{L}_{\infty}$ consisting of all leaves of $\fol_{\vert \overline{V}}$ carrying non-trivial holonomy has
Hausdorff dimension at most $2n-2$.
The lemma now follows from considering an exhaustion $\{ \overline{V}_n\}$ of $M \setminus {\rm Sing}\, (\fol)$, applying the preceding
conclusion to each compact set $\overline{V}_n$ in this exhaustion, and then taking the (countable) union of the corresponding sets.
\end{proof}

\section{Blowing up foliations and dicritical behavior}\label{Dicriticalnessoffoliations}

The central result of this section is Proposition~\ref{dicritical_divisors} showing that, up to an invariant ``small'' set, the notion
of ``enlarged leaf'' is, in a sense, nothing but a standard leaf considered in a suitable birational model for the foliation in question.
The whole issue hinges from the behavior of certain foliations under sequences of blow ups.

Recall that a {\it separatrix $\mathcal{C}$}\, for a foliation $\fol$ at a singular point $p$ is a germ of irreducible analytic
curve passing through~$p$ and such that $\mathcal{C} \setminus \{ p\}$ is contained in a regular leaf of $\fol$. Along similar lines,
we will say that a leaf $L$ of $\fol$ {\it induces a separatrix for $\fol$ at $p \in {\rm Sing}\, (\fol)$}, if there is a separatrix $\mathcal{C}$ through
a point $p$ such that $\mathcal{C} \setminus \{ p\}$ is contained in $L$. Clearly, a same leaf $L$
can induce only {\it countably many}\, separatrices at singular points of $\fol$.
In particular, the following assertions are equivalent:
\begin{itemize}
	\item The set of leaves inducing separatrices for $\fol$ has strictly positive Lebesgue measure (resp. Hausdorff dimension
	strictly greater than $2n-2$).

	\item Given a neighborhood $U \subset M$ of ${\rm Sing}\, (\fol)$, the set of separatrices of $\fol$ in $U$
	has strictly positive Lebesgue measure (resp. Hausdorff dimension strictly greater than $2n-2$).
\end{itemize}

%
%
%
%

In the sequel, we will work under the (weaker) condition involving Hausdorff dimension as emphasized throughout the paper.
As a basic remark often used in the following discussion, we note that
saturating (by $\fol$) transverse analytic sets of Hausdorff dimension at most $2n-4$ leads to $\fol$-invariant sets with
Hausdorff dimension bounded by $2n-2$. The same applies to countable unions of the former sets.

Now, Proposition~\ref{addingendstoleaves} reads as follows:

\begin{prop}
\label{addingendstoleaves}
Let $U \subset M$ be a neighborhood of ${\rm Sing}\, (\fol)$ and assume that the set of separatrices of $\fol$ in $U$ has
Hausdorff dimension strictly greater than $2n-2$. Then
there is a set $\mathcal{L}_0 \subset U$ of Hausdorff dimension strictly greater
than $2n-2$, invariant under $\fol$, and such that every leaf $L$ contained in $\mathcal{L}_0$ satisfies the following condition:
for every local branch of $L$ defining a separatrix
at a singular point $p$, this point~$p$ has to be added to the branch in question in order to form the enlarged leaf
$\widehat{L}$ containing $L$.
\end{prop}

Proposition~\ref{addingendstoleaves} is a consequence of the following local result:

\begin{prop}\label{dicritical_divisors}
Let $U \subset M$ be a neighborhood of ${\rm Sing}\, (\fol)$ and assume that the set ${\rm Sep}$ of separatrices of $\fol$ in $U$ has
Hausdorff dimension $\kappa > 2n-2$. Then, there is a set ${\rm Sep}_{\rm good} \subset {\rm Sep}$ such that the following holds:
\begin{itemize}
	\item[(1)] The Hausdorff dimension of ${\rm Sep}_{\rm bad} = {\rm Sep} \setminus {\rm Sep}_{\rm good}$ is at most~$2n-2$.

	\item[(2)] For every separatrix $\mathcal{C} \subset {\rm Sep}_{\rm good}$ at a point $p \in {\rm Sing}\, (\fol)$, the construction
	of the enlarged leaf containing $\mathcal{C}$ requires adding the point $p$ to $\mathcal{C}$.
\end{itemize}	
\end{prop}

\begin{proof}[Proof of Proposition~\ref{addingendstoleaves} after Proposition~\ref{dicritical_divisors}]
	Consider the saturated set of ${\rm Sep}_{\rm bad}$ by $\fol$. This
	saturated set is well defined since, bar the singular points themselves, every separatrix is contained in a unique leaf of $\fol$.
	However, since a same leaf of $\fol$ can induce several different separatrices of $\fol$ (at singular points that may or may not be different),
	the saturated set in question may contain separatrices lying in ${\rm Sep}_{\rm good}$. However, erasing all these
	separatrices from ${\rm Sep}$ amounts to erasing at most countably many copies of the initial set
	${\rm Sep}_{\rm bad}$. Therefore, the set of eliminated separatrices still has Hausdorff dimension bounded by $2n-2$.
	Thus there exists
	a subset ${\rm Sep}_{\rm exc} \subset {\rm Sep}_{\rm good}$ satisfying the following additional condition: whenever a leaf $L$ of $\fol$
	defines a separatrix lying in ${\rm Sep}_{\rm exc}$, then every other separatrix induced by the same leaf $L$ also lies in
	${\rm Sep}_{\rm exc}$. Furthermore, the Hausdorff dimension of ${\rm Sep}_{\rm exc}$ is still equal to~$\kappa > 2n-2$. Hence,
	the saturated by $\fol$ of ${\rm Sep}_{\rm exc}$ yields the desired invariant set and this completes the proof
	of Proposition~\ref{addingendstoleaves}.
\end{proof}

In view of the preceding, it suffices to prove Proposition~\ref{dicritical_divisors}. Let then $U$ and ${\rm Sep}$
be as in this proposition. Some specific terminology will be useful in the sequel.
Let $D$ denote a codimension~$1$ subvariety of $M$. A
local branch of a leaf $L$ of $\fol$ will be said to {\it land at a point $p \in D$} if the conditions below are satisfied:
\begin{itemize}
	\item The leaf $L$ is not contained in $D$.
	
	\item There exists an irreducible germ of analytic curve passing through~$p$ and contained in the branch in question.
\end{itemize}
Note that the point $p$ need not be singular for the foliation $\fol$, in which case the second item above would be equivalent
to saying that the branch in question induces a separatrix for $\fol$ at~$p$.
Similarly, we will say
that $p \in D$ is a {\it landing point}\, for $\fol$ if there is some leaf $L$ of $\fol$ landing at $p$. Note that the terminology of
``landing points'' only makes sense once a codimension~$1$ subvariety of $M$ is fixed.

Let us begin with a couple of elementary observations involving $\fol$ and codimension~$1$ subvarieties of $M$.
Assume that $D$ is an {\it irreducible}\, codimension~$1$ subvariety of $M$. Then, we have:
\begin{itemize}
	\item[(a)] If $D$ is not invariant by $\fol$ then, away from a proper analytic subset of $D$, the leaves of $\fol$ are
	transverse to $D$. In particular the set of landing points of $\fol$ lying in $D$ contains an open set
	$W \subset D$ having full ($(2n-2)$-dimensional) Lebesgue measure on $D$.

	\item[(b)] If $D$ is invariant by $\fol$ then the set of landing points in $D$ is contained in $D \cap {\rm Sing}\, (\fol)$
	which is a proper analytic subset of $D$. In particular, the subset of $D$ consisting of landing points of $\fol$ has Hausdorff dimension
	bounded by $2n-4$.
\end{itemize}

The proof of Proposition~\ref{dicritical_divisors} can further be reduced to the following proposition:

\begin{prop}\label{PlayingBlowup_Proposition}
	Every set contained in ${\rm Sep}$ and having Hausdorff dimension~$\kappa > 2n-2$, contains separatrices for which there
	exist foliated meromorphic immersions $f : \D^n \rightarrow M$ satisfying the following conditions:
	\begin{itemize}
		\item[(C.1)] $I (f) \cap (\{ 0\} \times \D)$ is reduced to the origin of $\C^n$.
		
		\item[(C.2)] The restriction of $f$ to $\{ 0\} \times \D$ is injective and $f (\{ 0\} \times \D)$ contains a relatively compact
		neighborhood of the singular point in the separatrix in question.
	\end{itemize}
\end{prop}

\begin{proof}[Proof of Proposition~\ref{dicritical_divisors} taking for granted Proposition~\ref{PlayingBlowup_Proposition}]
Let us assume for a contradiction that Proposition~\ref{dicritical_divisors} does not hold. Naturally we define
the set ${\rm Sep}_{\rm good} \subset {\rm Sep}$ as consisting of those separatrices whose landing points belong to the corresponding
enlarged leaf. In principle, ${\rm Sep}_{\rm good}$ may be empty. Now, we have ${\rm Sep}_{\rm bad} = {\rm Sep} \setminus {\rm Sep}_{\rm good}$
so that the Hausdorff dimension of ${\rm Sep}_{\rm bad}$ must be strictly greater than $2n-2$ since Proposition~\ref{dicritical_divisors} does not hold.

Next, owing to Lemma~\ref{Leaveswithoutholonomy}, up to deleting
a set of Hausdorff dimension at most~$2n-2$ from ${\rm Sep}_{\rm bad}$, all the remaining separatrices have trivial holonomy.
Summarizing, if Proposition~\ref{dicritical_divisors} does not hold, then there is a set ${\rm Sep}_{\rm bad'} \subset {\rm Sep}$ whose Hausdorff
dimension is $\kappa > 2n-2$ consisting of separatrices $\mathcal{C}$ satisfying the following conditions:
\begin{itemize}
	\item[(a)] If $\mathcal{C} \in {\rm Sep}_{\rm bad'}$, then the holonomy of $\mathcal{C}$ is trivial.
	
	\item[(b)] If $\mathcal{C} \in {\rm Sep}_{\rm bad'}$ is a separatrix at a singular point $p$, then $p$ does not belong to the enlarged leaf
	of $\fol$ containing $\mathcal{C}$.
\end{itemize}

To obtain the desired contradiction, we proceed as follows. According to Proposition~\ref{PlayingBlowup_Proposition},
there is a separatrix $\mathcal{C} \in {\rm Sep}_{\rm bad'}$ for which there exists 
a foliated meromorphic immersion $f : \D^n \rightarrow M$ satisfying both
conditions (C.1) and (C.2). Therefore $K = \overline{f (\{ 0\} \times \D)}$
is a vanishing end for the leaf of $\fol$ associated with $\mathcal{C}$,
see Section~\ref{Preliminaries}. Owing to Definition~\ref{defining_enlargedleaf}, the local branch corresponding to $\mathcal{C}$
of the enlarged leaf of $\fol$ containing $L$ must contain the singular point~$p$ as well. This contradicts the fact that
$\mathcal{C} \in {\rm Sep}_{\rm bad'}$ and establishes Proposition~\ref{dicritical_divisors}.
\end{proof}

In the remainder of this section, we will
focus on the proof of Proposition~\ref{PlayingBlowup_Proposition}. Let $U$ and ${\rm Sep}$ be as above. Consider a set
${\rm Sep}_{B} \subset {\rm Sep}$
having Hausdorff dimension equal to~$\kappa > 2n-2$ (strictly). We need to show
the existence of separatrices $\mathcal{C}$ in ${\rm Sep}_{B}$ along with foliated meromorphic immersions $f : \D^n \rightarrow M$
satisfying conditions~(C.1) and~(C.2) and such that $f (\{ 0\} \times \D) \subset \mathcal{C}$ is a relatively compact
neighborhood of the corresponding singular point in $\mathcal{C}$.

In what follows, all performed blow ups have smooth centers contained in the singular set of the foliation in question. 
Since ${\rm Sep}_{B}$ consists of separatrices of $\fol$, every element in ${\rm Sep}_{B}$, i.e., every germ of analytic curve
in ${\rm Sep}_{B}$, is naturally associated with a unique singular point of $\fol$, namely the point at which the corresponding
(germ of) analytic curve induces a separatrix for $\fol$.
Next, let $Z$ be an irreducible Zariski-closed subset of ${\rm Sing}\, (\fol)$ satisfying the
two conditions below:
\begin{itemize}
	\item[(i)] The subset of ${\rm Sep}_{B}$ consisting of separatrices associated with singular
	points lying in $Z$ has Hausdorff dimension $\kappa > 2n-2$.
	
	\item[(ii)] For every proper Zariski-closed subset $Z'$ of $Z$, the set of separatrices in ${\rm Sep}_{B}$ associated with
	singular points lyting in $Z'$ has Hausdorff dimension at most $2n-2$.
\end{itemize}
Since ${\rm Sing}\, (\fol) \subset M$ is a proper analytic subset (of codimension~$\geq 2$), the existence of a Zariski-closed set
$Z \subset {\rm Sing}\, (\fol)$ as indicated is clear. Without loss of generality (Hironaka resolution), we can assume that $Z$ is smooth. Moreover,
since $Z \subset {\rm Sing}\, (\fol)$, the codimension of $Z$ is at least~$2$ so that it can be used as a blow up center. To abridge notation,
in the sequel we assume without loss of generality that all separatrices in ${\rm Sep}_{B}$ are associated with points in $Z$.

Now, consider the blow up $M \stackrel{\Pi_{(1)}}{\longleftarrow} M_1$ of $M$ centered at $Z$.
Lemma~\ref{Lemma_transformingseparatrices} below shows how the separatrices in ${\rm Sep}_{B}$ transform under $\Pi_{(1)}$.
In this lemma we use the terminology of landing points with respect to the codimension~$1$ subvariety $D$ coinciding with the
exceptional divisor $D = \Pi_{(1)}^{-1} (Z)$.

\begin{lemma}\label{Lemma_transformingseparatrices}
Without loss of generality, we can assume that the transform of every separatrix $\mathcal{C} \in {\rm Sep}_{B}$ lands at a point
in $D = \Pi_{(1)}^{-1} (Z)$ that is singular for $\fol_1 = (\Pi_{(1)}^{-1})^{\ast} \fol$. In particular, $D = \Pi_{(1)}^{-1} (Z)$
can be assumed to be invariant by $\fol_1$.
Moreover, there exists
an irreducible Zariski-closed subset $Z_1 \subset {\rm Sing}\, (\fol_1)$ satisfying conditions~(i) and~(ii) above and verifying
$\Pi_{(1)} (Z_1) = Z$. Finally, the dimension ${\rm dim}\, (Z_1)$ of $Z_1$ is greater than or equal to the dimension of $Z$.
\end{lemma}

\begin{proof}
Recall that separatrices are irreducible by definition. Thus, the transform $\mathcal{C}_1$ of each separatrix $\mathcal{C}$ in
${\rm Sep}_{B}$ lands at a well defined (unique) point $q$ in the exceptional divisor $\Pi_{(1)}^{-1} (Z)$. Clearly $\Pi_{(1)} (q)$ lies in
$Z$. Next, if the transform of some separatrix in ${\rm Sep}_{B}$ lands at a regular point $q$ of $\fol_1$, then there are foliated
coordinates around $q$ yielding a diffeomorphism from $\D^n$ to $M_1$ (in particular the exceptional divisor
$\Pi_{(1)}^{-1} (Z)$ is not invariant by $\fol_1$).
As in Lemma~\ref{blowup_invariance}, once the mentioned diffeomorphism is composed
with the projection $\Pi_{(1)}$, we obtain the desired foliated immersion proving Proposition~\ref{PlayingBlowup_Proposition}.

Thus we can assume that the transforms of all separatrices in ${\rm Sep}_{B}$ land at points of $D = \Pi^{-1} (Z)$ that are
singular for $\fol_1$.
The existence of the irreducible component $Z_1$ satisfying conditions~(i) and~(ii) now follows from repeating the argument
yielding the existence of the initial $Z$ (up to reducing the set ${\rm Sep}_{B}$ without affecting its Hausdorff dimension).
The fact that separatrices landing at $Z_1$ project to separatrices in ${\rm Sep}_{B}$ associated with points in $Z$
implies that $\Pi_{(1)} (Z_1) = Z$ and also
yields ${\rm dim}\, (Z_1) \geq {\rm dim}\, (Z)$. The lemma is proved.
\end{proof}

Beginning with the irreducible set $Z$ and the blow up centered at $Z$, we can now blow up $Z_1$ and continue by induction.
Thus we can consider arbitrary sequences of blow ups
\begin{equation}
(M,\fol, Z) = (M_0, \fol_0, Z_0) \stackrel{\Pi_{(1)}}{\longleftarrow} (M_1, \fol_1, Z_1) \stackrel{\Pi_{(2)}}{\longleftarrow} \cdots
\stackrel{\Pi_{(k)}}{\longleftarrow} (M_k, \fol_k, Z_k) \label{sequenceofblowups}
\end{equation}
such that all of the following hold:
\begin{itemize}
	\item[1.] $\Pi_{(j)}$ is the blow up of $M_{j-1}$ centered at $Z_{j-1}$, $j=1, \ldots , k$.
	
	\item[2.] For every $j=1, \ldots, k$, $Z_j \subseteq {\rm Sing}\, (\fol_j)$ and $\Pi_{(j)} (Z_j) = Z_{j-1}$.
	
	\item[3.] For every subset $S= S_0 \subseteq {\rm Sep}_{B}$ with Hausdorff dimension equal to $\kappa > 2n-2$, there
	are sets $S_j$ fulfilling the following conditions:
	   \begin{itemize}
	    \item[3.A] The set $S_j$ consists entirely of separatrices of $\fol_j$ landing at $Z_j$.
	    
	    \item[3.B] For every $j=1, \ldots, k$, $\Pi_{(j)} (S_j) \subseteq S_{j-1}$.
	   \end{itemize}
\end{itemize}

Since ${\rm dim}\, (Z_i) \geq {\rm dim}\, (Z_{i-1})$ for every $i=1,2, \ldots$, the sequence $\{ {\rm dim}\, (Z_j)\}_{j}$  eventually
becomes constant. Thus there is no loss of generality in assuming that
the centers $Z = Z_0, \ldots , Z_k, \ldots$ have all the same dimension.

We are now ready to prove Proposition~\ref{dicritical_divisors}. Again we will use the terminology of ``landing points'' with respect
to the (total) exceptional divisor.

\begin{proof}[Proof of Proposition~\ref{dicritical_divisors}]
To prove the proposition we will show that the situation described in the above items~1, 2, and~3 cannot occur for every sequence
of blow ups. More precisely, it will be seen that under a suitable sequence of blow ups, there are separatrices in ${\rm Sep}_{B}$
whose transforms will land at points of the corresponding exceptional divisor that are regular for $\fol_k$.

Let us first consider the prototypical case where all the centers $Z_j$ are reduced to points (i.e., each $Z_j$ is a discrete
set of points). In this case, we set $Z = P_0 \in {\rm Sing}\, (\fol)$
and consider the separatrices in ${\rm Sep}_{B}$. These separatrices are all associated with $P_0$. By definition,
separatrices are irreducible and, up to fixing local coordinates around $P_0$, we can assume none of them is contained in
a Cartesian hyperplane. In fact, separatrices contained in the union of the Cartesian hyperplanes form a set of Hausdorff dimension
bounded by $2n-2$ so that they can be deleted from ${\rm Sep}_{B}$ without loss of generality. Also, being irreducible, 
each separatrix $\mathcal{C}$
in ${\rm Sep}_{B}$ possesses a one-to-one Puiseux parameterization $\mathcal{P}_{\mathcal C}$ none of whose coordinates vanishes identically
in view of the preceding. Thus, every separatrix $\mathcal{C}$ in ${\rm Sep}_{B}$ has a Puiseux parameterization of
the form $\mathcal{P}_{\mathcal C} (s) = (s^{l_1} h_1, \ldots , s^{l_n} h_n)$ where $l_i$ is a positive integer and where $h_i$ is a
non-vanishing holomorphic function around $0 \in \C$, $i=1, \ldots, n$. Now, we have the following:

\vspace{0.1cm}

\noindent {\it Claim}. Without loss of generality, we can assume that the $n$-tuple of integers $(l_1, \ldots , l_n)$
does not depend on the choice of the separatrix $\mathcal{C}$.

\vspace{0.1cm}

\noindent {\it Proof of the Claim}. Since the $l_i$ are (positive) integers, there are only countably many possible $n$-tuples.
Since a countable union of sets with Hausdorff dimension strictly less than $\kappa$ still has Hausdorff dimension strictly less
than $\kappa$, there follows the existence of a subset $S= S_0 \subseteq {\rm Sep}_{B}$ of Hausdorff dimension $\kappa$ and such that
all the (irreducible) Puiseux parameterizations of separatrices in $S=S_0$ share the same integers $l_1, \ldots , l_n$.\qed

Sharing the same $n$-tuple $(l_1, \ldots , l_n)$, all these separatrices become smooth at the same blow up. In other words, the preceding
allows us to assume without loss of generality that the separatrices in ${\rm Sep}_{B}$ are all smooth.
In particular, if the situation described in items~1, 2, and~3 were to occur for every sequence of blow ups,
there would follow the existence of a set of separatrices possessing Hausdorff dimension strictly greater than $2n-2$
all of whose elements would always land at the same point, regardless of the sequence of blow ups chosen. Note that the last assertion
holds thanks to the fact that we are currently working under the additional assumption that every center
$Z_j$ consists of a discrete set of points.
In fact, a countable union of sets with Hausdorff dimension bounded by $2n-2$ still has Hausdorff dimension bounded by $2n-2$. In other words,
for every $j$, there must exist at least one point in $Z_j$ at which the transforms of the separatrices belonging to a certain
subset of Hausdorff dimension $\kappa > 2n-2$ of ${\rm Sep}_{B}$ land. Now,
separatrices in this set
necessarily have infinite contact so that the set in question is reduced to a single separatrix. This yields a contradiction since
the Hausdorff dimension of this set is strictly greater than $2n-2 \geq 2$.

Now that we have established Proposition~\ref{dicritical_divisors} in the special case where all the centers $Z_j$ are reduced
to points, let us consider the general case. Recall that all the centers $Z=Z_0, \ldots, Z_k, \ldots$ have the same dimension.
Set ${\rm dim}\, (Z) = {\rm dim}\, (Z_0) = \cdots = {\rm dim}\, (Z_k) = m \leq n-2$. To abridge notation,
the initial set ${\rm Sep}_{B}$ will be identified with its transforms in the following discussion.
Without loss of generality, we can also
assume that the exceptional divisor is not empty so that the terminology ``landing points'' can again be used with respect
to the (total) exceptional divisor. Now, the center $Z_j$ 
is a codimension~$n-m \geq 2$ submanifold of $M_j$ and this holds for every~$j$. With the mentioned identifications, all separatrices
in ${\rm Sep}_{B}$ land at points in $Z_j$.

In the sequel, we shall make several ``generic'' choices. In all these instances, the word ``generic'' should be
understood as lying in the complement of {\it countably many proper Zariski-closed subsets}. In particular, these choices are
also generic in terms of the Lebesgue measure of corresponding dimension. With this terminology in place, we proceed as follows.
For every point $p \in Z_j$, let
${\rm Sep}_{B}^{(p)}$ denote the subset of ${\rm Sep}_{B}$ consisting of separatrices landing at~$p \in Z_j$.
Since sets of separatrices landing at proper Zariski-closed
subsets of $Z_j$ have Hausdorff dimension at most $2n-2$, the theorem will follow once we prove that, for a ``generic''
point $p$ of $Z_j$, the set ${\rm Sep}_{B}^{(p)}$ is transversely parameterized by an analytic set $\mathcal{S}^{\pitchfork}_p$ whose
complex dimension is at most $n-m-2$. In other words, the set ${\rm Sep}_{B}^{(p)}$ being the saturated by $\fol$ of the
analytic set $\mathcal{S}^{\pitchfork}_p$, the dimension of ${\rm Sep}_{B}^{(p)}$ is at most $n-m-1$. If this is so, then Fubini theorem
applied to the set
$$
\bigcup_{p \in Z_j} {\rm Sep}_{B}^{(p)}
$$
implies that the set in question has (non-zero) locally finite $d$-dimensional mass in some real dimension~$d \leq 2n-2$. Therefore
this value of~$d$ must coincide with the Hausdorff dimension of the set in question thus contradicting our assumption that
this Hausdorff dimension equals~$\kappa > 2n-2$.

To prove that ${\rm Sep}_{B}^{(p)}$ is transversely parameterized by a complex analytic set $\mathcal{S}^{\pitchfork}_p$ whose (complex) dimension is at most 
$n-m-2$ is slightly more involved than in the previous situation where $Z_j$ were reduced to points. The argument is detailed below.

Around a ``generic'' point $p$ of $Z_j$,
there are coordinates $(x_1, \ldots ,x_n)$ such that $Z_j = \{ x_{m+1} = \cdots = x_n =0\}$ and then separatrices in 
${\rm Sep}_{B}$ can again be parameterized by
$$
s \mapsto (s^{l_1} h_1, \ldots , s^{l_n} h_n) + (p_1, \ldots, p_m, 0, \ldots ,0) \, .
$$
Here $p \in Z_j$ is identified with $(p_1, \ldots, p_m, 0, \ldots ,0)$.
The set $\bigcup_{p \in Z_j} {\rm Sep}_{B}^{(p)}$ can be partitioned in countably many sets encoded by the $n$-tuples
$(l_1, \ldots, l_n)$ of positive integers. Using again the fact that a countable union of sets having Hausdorff dimension
strictly less than $\kappa$ will still have Hausdorff dimension strictly less than $\kappa$, we can therefore assume without
loss of generality that all the separatrices in $\bigcup_{p \in Z_j} {\rm Sep}_{B}^{(p)}$ share the same set of
``exponents'' $(l_1, \ldots ,l_n)$.

Once again, up to choosing suitable local coordinates $(x_1, \ldots ,x_n)$, we can assume that the separatrices in ${\rm Sep}_{B}^{(p)}$
are not contained in any of the coordinate hyperplanes given by $\{ x_i = 0\}$ for $i=m+1, \ldots, n$. In other words, we can assume
without loss of generality that none of the function
$h_{m+1}, \ldots, h_n$ vanish at $p$ (i.e., the corresponding coordinates of their Puiseux parameterizations do not vanish identically).
Up to systematically
performing blow ups centered at $Z_j$, the first $m$ coordinates of the parameterization are preserved,
up to the natural identifications. The remaining $n-m$ coordinates, related to the direction transverse to $Z_j$ are affected
in the standard way. These separatrices will become smooth because the derivative at $s=0$ of
one of these ``last'' $n-m$-coordinates will eventually be different from zero. The sequence of blow ups only depends on the
exponents $l_{m+1}, \ldots , l_n$. Thus, arguing as above, all these separatrices become smooth at a same blow up.
Now, the contact between separatrices {\it landing at the same point}\, of $Z$ is defined taking in consideration
the normal bundle of $Z_j$. In more accurate terms, we define the contact taking into consideration only the $n-m$ ``last''
coordinates. They can therefore be pulled apart and, since none of these blow ups is dicritical (by assumption), they
are parameterized by analytic sets of complex dimension at most $n-m-2$ (the dimension of the exceptional divisor over one point in the
previous center minus 1). In fact, if $\Pi$ denotes a blow up centered at a $m$-dimensional submanifold, then the pre-image
$\Pi^{-1} (p)$ of a point $p$ in this center has dimension $n-m-1$.

The last statement falls slightly short of proving Proposition~\ref{dicritical_divisors}. The reason lies in the fact that,
a priori, only separatrices having finite contact {\it with respect to the normal bundle}\, can be pulled apart by the systematic
blow up of centers $Z_j$ as above. To clarify this statement, consider again sufficiently generic local coordinates
$(x_1, \ldots ,x_n)$ with $Z_j = \{ x_{m+1} = \cdots = x_n =0\}$. In view of the preceding, we can assume that all separatrices
in ${\rm Sep}_{B}^{(p)}$ are smooth and admit a Puiseux parameterization of the form
$$
s \mapsto (\varphi_1 (s), \ldots , \varphi_{n-1}(s), s) \, 
$$
where $p$ is identified with the origin of $\C^n$. Let then $\mathcal{C}_1$ and $\mathcal{C}_2$ be two separatrices in ${\rm Sep}_{B}^{(p)}$
with Puiseux parameterizations respectively given by
$$
(\varphi_1^{(1)} (s), \ldots , \varphi_{n-1}^{(1)}(s), s) \; \; \; {\rm and} \; \; \;
(\varphi_1^{(2)} (s), \ldots , \varphi_{n-1}^{(2)}(s), s) \, .
$$
Note that the separatrices $\mathcal{C}_1$ and $\mathcal{C}_2$ can be
taken apart by a sequence of blow ups of the codimension~$n-m$ singular set $Z_j = \{ x_{m+1} = \cdots = x_n =0\}$ unless
we have $\varphi_i^{(1)} = \varphi_i^{(2)}$ for every $i=m+1, \ldots, n-1$. Therefore, Proposition~\ref{dicritical_divisors}
follows at once from Lemma~\ref{onemorelemma_rightbelow} below.
\end{proof}

\begin{lemma}\label{onemorelemma_rightbelow}
	With the preceding notation, for a generic point $p \in Z_j$ and a generic separatrix $\mathcal{C}$ in ${\rm Sep}_{B}^{(p)}$,
there are at most countably many separatrices in ${\rm Sep}_{B}^{(p)}$ that cannot be pulled apart from $\mathcal{C}$ by means
of successive blow ups centered at the $m$-dimensional singular sets $Z_j$.
\end{lemma}

\begin{proof}
We fix a generic separatrix $\mathcal{C}$ in ${\rm Sep}_{B}^{(p)}$ as in the statement.
The corresponding foliation is denoted by $\fol_j$.
To abridge notation, we can assume that $\mathcal{C}$ coincides with the $x_n$-axis so that it has a natural Puiseux parameterization
given by $s \mapsto (0, \ldots, 0, s)$.
We need to prove that the set ${\rm Sep}_{B}^{(p,x_n)} \subset {\rm Sep}_{B}^{(p)}$ consisting of separatrices in ${\rm Sep}_{B}^{(p)}$
with a Puiseux parameterization of the form $s \mapsto (\varphi_1, \ldots, \varphi_m,0, \ldots, 0, s)$ contains at most countably
many elements (indeed, finitely many elements).

For this, let us first consider a vector field $\mathcal{Y}$ representing the foliation $\fol_j$ around the point $p$. In other words,
$\mathcal{Y}$ is a holomorphic vector field with zero-set of codimension at least~$2$ which is tangent to $\fol_j$. Next, we let
\begin{equation}
\mathcal{Y} = F_1 \frac{\partial}{\partial x_1} + \cdots + F_n \frac{\partial}{\partial x_n} \, . \label{definingYideals}
\end{equation}
Since $\fol_j$ is singular over $Z_j = \{ x_{m+1} = \cdots = x_n =0\}$, all the coordinate functions $F_i$, $i=1, \ldots, n$,
belong to the ideal $\mathcal{I}\, (x_{m+1}, \ldots, x_n)$ generated by the coordinates $x_{m+1}, \ldots, x_n$. Furthermore,
since the $x_n$-axis is a separatrix for $\fol_j$, the holomorphic function $s \mapsto F_n (0, \ldots, 0, s)$ vanishes at
$s=0$ but it is not identically zero while all the functions $s \mapsto F_i (0, \ldots, 0, s)$ are constant equal to zero for $i=1,\ldots, n-1$.
The order at $s=0$ of the function $s \mapsto F_n (0, \ldots, 0, s)$ is called {\it the multiplicity of $\fol_j$ along the separatrix
$\mathcal{C} \simeq \{ x_1= \cdots = x_{n-1} = 0\}$}, see \cite{helenajulio_RMS}, page~311.

The interest of this notion of multiplicity of a foliation along a separatrix lies in the fact that they form a decreasing sequence
under blow ups so that it eventually becomes stationary. More precisely, consider first the family of automorphisms of $\C^n$ given by
$\Lambda (x_1, \ldots ,x_n) = (x_1, \ldots ,x_m, \lambda x_{m+1}, \ldots , \lambda x_n)$, with $\lambda \in \C^{\ast}$. For
the blow ups centered at the submanifolds $Z_j$ as above, it is useful to define the {\it order of $\fol_j$ with respect
to the center $Z_j$}\, as the unique positive integer $d$ for which
\begin{equation}\label{homogeneous_component}
\lim_{\lambda \rightarrow 0} \frac{1}{\lambda^{d-1}} \Lambda^{\ast} X
\end{equation}
yields a non-trivial vector field. In the present case where $\mathcal{Y}$ is singular over $Z_j = \{ x_{m+1} = \cdots = x_n =0\}$
the order~$d$ is strictly positive and it equals~$1$ if and only if the vector field
$$
\mathcal{Y}_2 = F_{m+1} \frac{\partial}{\partial x_{m+1}} + \cdots + F_n \frac{\partial}{\partial x_n} 
$$
has order~$1$ at the origin of $\C^{n-m}$ viewed as a vector field in the variables $x_{m+1}, \ldots, x_n$ with coefficients
in $\C [x_1, \ldots ,x_m]$.

When the sequence formed by the multiplicity of a foliation along the separatrix
$\mathcal{C} \simeq \{ x_1= \cdots = x_{n-1} = 0\}$ becomes stationary, then the order of $\fol_j$ with respect
to the center $Z_j$ is necessarily equal to~$1$, see Section~3 of \cite{helenajulio_RMS} and note that, whereas \cite{helenajulio_RMS}
deals with the case of $3$-dimensional ambient manifolds, the proof given there applies equally well to the present case.
Thus, up to performing finitely many additional blow ups of the centers
$Z_j$, we can assume without loss of generality that the order of $\fol_j$ with respect
to $Z_j$ is always equal to~$1$. Now, given that $p$ can be chosen generic inside $Z_j$, we can assume that
the ``polynomial coefficients'' of $\mathcal{Y}_2$ in $\C [x_1, \ldots ,x_m]$ do not vanish at~$p$''.

Summarizing the two preceding paragraphs, with $p$ identified with the
origin of $\C^n$ as before, we can assume that the Taylor series of $\mathcal{Y}$ at the origin begins with a degree~$1$
(homogeneous) vector field having the form
\begin{equation}
\mathcal{Y}^{(1)} = H_1 \frac{\partial}{\partial x_1} + \cdots + H_n \frac{\partial}{\partial x_n} \label{linearpartofY}
\end{equation}
where the coordinate functions $H_i$, $i=1, \ldots, n$, satisfy all of the following:
\begin{itemize}
	\item Every $H_i$ is a linear form with constant coefficients in the variables $x_{m+1}, \ldots, x_n$.

	\item At least one of the forms $H_{m+1}, \ldots , H_n$ does not vanish identically (but the remaining forms
	$H_1, \ldots , H_m$ might vanish identically).

\end{itemize}

We are now ready to derive the statement of Lemma~\ref{onemorelemma_rightbelow}. We will
show that ${\rm Sep}_{B}^{(p,x_n)}$ contains only countably many elements.
Assume aiming at a contradiction that the statement is false. Clearly all the separatrices in ${\rm Sep}_{B}^{(p,x_n)}$ can
be pulled apart by systematically performing one-point blow ups. Now, the contact between separatrices in ${\rm Sep}_{B}^{(p,x_n)}$ is defined in
the standard way. Arguing as in the case where all the centers $Z_j$ were reduced
to points, there follows the existence of uncountably many separatrices
in ${\rm Sep}_{B}^{(p,x_n)}$ which have pairwise contact of order $s \in \N$, for some fixed positive integers~$s$. Denote by
$\mathcal{P}_s$ the set formed by these separatrices. Since
$\mathcal{C} \simeq \{ x_1= \cdots = x_{n-1} = 0\}$ is chosen ``generic'', we can assume that $\mathcal{C}$ belongs to $\mathcal{P}_s$,
i.e., the separatrices in question have contact of order~$s$
with $\mathcal{C}$ as well.
Because the contact between these separatrices fall by one unity with each blow up, we can assume to abridge the discussion that
$s=1$ since the general case follows by a straightforward induction argument.

To obtain a contradiction proving Lemma~\ref{onemorelemma_rightbelow}, let us perform the blow up of $\fol_j$
centered at~$p$. Let $\tilf_j$ denote the corresponding transform of $\fol_j$. Each separatrix in $\mathcal{P}_s$ defines
a singular point for $\tilf_j$ in the exceptional divisor $E \simeq \CP^{n-1}$. The Zariski-closure of these singular points in $E$
defines a proper algebraic set $A$ of $E \simeq \CP^{n-1}$ containing the origin of the standard affine coordinates
$z_1, \ldots , z_n$ for the blow up of $\C^n$, with $z_n =x_n$. The set $A$ has strictly positive dimension since
$\mathcal{P}_s$ contains uncountably many separatrices.

In slightly more precise terms, note that $A$ is contained in
the projective subspace of $E \simeq \CP^{n-1}$ associated with the span of $x_1, \ldots, x_m, x_n$ (the tangent space of $Z_j$
plus the direction of $x_n$). In particular, the intersection of the transform of $Z_j$ and $A$ is not empty.

Next, note that the behavior of $\fol_j$ with respect to the set $A$
is {\it dicritical}\, in the sense that all points in $A$ are singular for $\tilf_j$. The last statement can be made more
accurate as follows. Denote by $\mathcal{A} \subset \C^n$ the analytic set of $\C^n$ which is invariant by the radial vector field
of $\C^n$ and whose tangent cone at the origin is identified with~$A \subset E \simeq \CP^{n-1}$. The set $\mathcal{A}$ is unequivocally
defined and it is invariant by the linear vector field $\mathcal{Y}^{(1)}$ in~(\ref{linearpartofY}). The restriction to $\mathcal{A}$ of $\mathcal{Y}^{(1)}$
yields then a dicritical foliation in the usual sense. However, the latter foliation being
dicritical, there follows that restricted to $\mathcal{A}$ the vector field $\mathcal{Y}^{(1)}$ is a multiple of the radial vector field
and, in fact, it must be a constant multiple of the radial vector field since $\mathcal{Y}^{(1)}$ has degree~$1$. This yields the desired
contradiction since $\mathcal{Y}$ is singular on $Z_j = \{ x_{m+1} = \cdots = x_n =0\}$ (c.f., Equation~\ref{definingYideals})
and the transform of $Z_j$ intersects $A$ non-trivially.
The proof of Lemma~\ref{onemorelemma_rightbelow} is completed.
\end{proof}

\section{Meromorphic solutions, parabolic leaves, and invariant surfaces}\label{meromorphicsolutions_brunellatheorem}

Throughout this section, $M$ will stand for a complex projective manifold of dimension~$n$ while $X$ will denote a meromorphic
vector field defined on $M$. The divisor of poles of $X$ is denoted by $(X)_{\infty}$ and the foliation on $M$ associated with
$X$ is denoted by $\fol$.

The purpose of this section is to prove the theorem below:

\begin{theorem}
	\label{classification-foliation}
	Let $M$, $X$, and $\fol$ be as above and assume the existence of a
	foliated set $\mathcal{M} \subset M \setminus (X)_{\infty}$ having Hausdorff dimension strictly greater than $2n-2$ and such that for every point
	$p \in \mathcal{M}$, the corresponding solution of $X$ yields a meromorphic map defined on all of $\C$. Then, for every $\fol$-invariant
	irreducible surface $N \subset M$ (not contained in the singular set of $\fol$), the foliation $\fol_N$ induced on $N$ by $\fol$
	is birationally equivalent to one of the following types of foliations:
	\begin{enumerate}
		\item A fibration.
		\item A Riccati foliation.
		\item A turbulent foliation.
	\end{enumerate}
\end{theorem}

As will be seen, Theorem~\ref{classification-foliation} is the main ingredient in the proof of Theorem~A. In turn, the
proof of Theorem~\ref{classification-foliation} involves some fine material from complex geometry/analysis. First, however, we shall
introduce some terminology needed to clarify the content of the theorem in question. Let then $N$ denote
a (possibly singular) compact complex surface equipped
with a singular holomorphic foliation $\fol$. The pair $(N, \fol)$ will be regarded up to birational equivalence as
it is standard in complex geometry.

Recall that a {\it fibration}\, on a complex surface~$N$ is nothing but a non-constant holomorphic map with connected
fibers $\calp : N \rightarrow S$ where $S$ stands for some compact Riemann surface. Clearly the set of critical values of
$\calp$ is a finite subset $\{ p_1, \ldots ,p_l\}$ of $S$ and the fibers of $\calp$ over the critical values are said to be
the {\it singular fibers}\, of the fibration (or of $\calp$). Note also that the map $\calp$ induces a $\mathcal{C}^{\infty}$-bundle
structure from $M \setminus \bigcup_{i=1}^l \calp^{-1} (p_i)$ to $S \setminus \{ p_1, \ldots ,p_l\}$. The genus of the corresponding
fiber is called the genus of $\calp$. Thus a fibration is of genus~$0$ if and only if its generic fiber is a rational curve.
Similarly, saying that the genus of the fibration is~$1$ is tantamount to saying that
its generic fiber is an elliptic curve.

Naturally a fibration is a special type of foliation. Next, assume that $N$ is equipped with a fibration $\calp : N \rightarrow S$
and with another singular foliation $\fol$. The foliation $\fol$ is said to be {\it transverse to $\calp$}\, if, away from finitely many
$\fol$-invariant fibers, the leaves of $\fol$ are transverse to the fibers of $\calp$.
Finally, if $\fol$ is transverse to some genus~$0$ fibration, then $\fol$ is called a {\it Riccati foliation}. Similarly a
foliation transverse to a genus~$1$ fibration is called a {\it turbulent foliation}.

A last comment is needed to make accurate what is the foliation $\fol_N$ induced by $\fol$ on $N$. Whereas $N$ is invariant
by $\fol$, the issue hinges from the fact that
our foliations are singular so that the notion of ``foliation induced by $\fol$ on $N$'' may, strictly speaking, differ from the
restriction of $\fol$ to $N$. To explain the difference, recall that $N$ has complex dimension~$2$ so that, by definition, every
holomorphic singular foliation on $N$ has only isolated singular points. The singular set ${\rm Sing}\, (\fol)$
of $\fol$, however, is only assumed to have codimension~$2$ so that the intersection $N \cap {\rm Sing}\, (\fol)$ may have
components of {\it dimension~$1$}\, (note that $N \cap {\rm Sing}\, (\fol)$ is a proper closed subset of $N$ since by assumption $N$ is not contained
in ${\rm Sing}\, (\fol)$). These ``curves of singular points in $N$'' must hence be eliminated by using Hilbert nullstellensatz
in codimension~$1$ so as to produce a singular holomorphic foliation on~$N$
according to Definition~\ref{foliation_definition}, c.f., the proof of
Lemma~\ref{lemma_meromorphicvfieldfoliation}. This gives rise to the following:

\begin{defi}
	\label{inducedfoliation_definition}
The singular holomorphic foliation on $N$ obtained after eliminating curves of singular points in $N \cap {\rm Sing}\, (\fol)$ (in the above
mentioned sense) is called the foliation induced by $\fol$ on $N$ and denoted by $\fol_N$.
\end{defi}
The notion of induced foliation, as opposed to restriction of a foliation, actually plays a role in the
proof of Theorem~\ref{classification-foliation}, albeit a minor one. This will stem from the
difference between a leaf of $\fol$ contained in $N$ and a
leaf of $\fol_N$: the latter being obtained from the former by (possibly) adding points lying in $N \cap {\rm Sing}\, (\fol)$.
Indeed, a point in $N \cap {\rm Sing}\, (\fol)$ may turn out to be {\it regular}\, for $\fol_N$. Also,
away from $N \cap {\rm Sing}\, (\fol)$, the leaves of $\fol_N$ locally coincide with the leaves of $\fol$ contained
in $N$.

We are now able to begin our approach to Theorem~\ref{classification-foliation}.
Let $M$, $X$, and $\fol$ be as in this theorem. According to standard terminology, a Riemann surface is said to
be {\it parabolic}\, if it is a quotient of $\C$. Similarly, a Riemann surface that is a quotient of the hyperbolic disc $\D$ is called
{\it hyperbolic}. Owing to the uniformization theorem, every Riemann surface different from $\Cpp$
is either parabolic or hyperbolic. The key step is the following proposition:

\begin{prop}
\label{No_hyperbolic-leaf}
Every enlarged leaf of $\fol$ is either $\Cpp$ or a parabolic Riemann surface.
\end{prop}

The point of considering enlarged leaves $\widehat{L}$, as opposed to standard leaves $L$, lies in the fact
that the Poincar\'e metric on {\it enlarged leaves}\, varies in a pluri-subharmonic way, thanks to Brunella's
theorem in \cite{subharmonic-variation}, see also \cite{Ivaskovich} and
\cite{Brunella-notes}. This means the following. Let every enlarged leaf
of $\fol$ be equipped with its Poincar\'e metric: the unique complete metric with curvature equals to~$-1$, in the
case of a hyperbolic Riemann surface, and the totally vanishing ``metric'' otherwise. If $Z$ is a (locally defined)
holomorphic vector field tangent to $\fol$, then we have the (locally defined) function $P$ given by
$$
P (q) = \log (\Vert Z (q) \Vert)
$$
where $\Vert \, . \, \Vert$ stands for the norm of $X$ with respect to the above introduced foliated metric. Basically,
Brunella's theorem states that the function $P$ is {\it plurisubharmonic}\, unless it is constant equal to~$-\infty$.
Clearly the polar set of $P$ coincides with non-hyperbolic enlarged leaves of $\fol$. Therefore we obtain the
following corollary of Brunella's theorem:
if the set of {\it enlarged leaves}\, of $\fol$ that {\it are not} hyperbolic Riemann surface is too large to be a polar set,
then {\it every enlarged leaf of $\fol$ is either parabolic or a rational curve}.

In view of what precedes and of \cite{Who??}, the proof of Proposition~\ref{No_hyperbolic-leaf} becomes 
reduced to the proposition below:

\begin{prop}\label{Summarized-No_hyperbolic-leaf}
	The set of leaves of $\fol$ giving rise to enlarged leaves that are either parabolic Riemann surfaces or rational curves
	has Hausdorff dimension strictly greater than $2n-2$.
\end{prop}

To begin the proof of Proposition~\ref{Summarized-No_hyperbolic-leaf},
consider a leaf $L$ of $\fol$ that is meromorphically parameterized by $\C$ via $X$. We denote by $\widehat{L}$
the enlarged leaf containing $L$. By assumption, the set of leaves $L$ meromorphically parameterized by $\C$ via $X$ has Hausdorff
dimension strictly larger than $2n-2$.

Assume first that the parameterization $\phi$ is actually holomorphic on all of $\C$. Then $L$ cannot be a hyperbolic Riemann surface
since there are non-constant holomorphic maps from $\C$ to $L$. The same conclusion clearly applies to $\widehat{L}$ since
$L \subseteq \widehat{L}$. In view of the preceding, up to ignoring a set of leaves having Hausdorff dimension at most $2n-2$,
we can assume without loss of generality that all meromorphic solutions $\phi : \C \rightarrow M$ of $X$ have non-empty
polar set. In other words, $\phi$ is strictly meromorphic.

At this point, it is convenient to recall the general definition of meromorphic maps between complex spaces along with some of their basic properties,
cf. \cite{remmert}. Let $A,B$ be {\it normal} complex spaces and consider an analytic set $C$ strictly contained in $A$. Let $F : A\setminus C
\rightarrow B$ be a holomorphic map. Then $F$ is said to be a {\it meromorphic map from $A$ to $B$} if the closure
$\overline{{\rm Graph}\, (F)} \subset A \times B$ of the graph of $F$
is an analytic subset of $A \times B$ and the restriction to $\overline{{\rm Graph}\, (F)} \subset A \times B$ of the
natural projection $\pi_A: A \times B \rightarrow A$ is proper. Under these conditions, the set $\overline{{\rm Graph}\, (F)}$
is called the {\it graph of the meromorphic map $F : A \rightarrow B$}. The projection $\pi_A : \overline{{\rm Graph}\, (F)}
\rightarrow A$ is clearly surjective. If $U \subset A$ is the largest open set on which $F$ can be extended as a holomorphic
map, then $I = A \setminus U$ is an analytic subset of $A$ with codimension at least~$2$
called the {\it indeterminacy set of $F$}. The set $\pi_A^{-1} (U)$
is open and dense in $\overline{{\rm Graph}\, (F)}$ and $\overline{{\rm Graph}\, (F)} \setminus \pi_A^{-1} (U)$
is a proper analytic subset of $\overline{{\rm Graph}\, (F)}$. Finally, a point $p \in A$ whose pre-image
$\pi_A^{-1} (p)$ has strictly positive dimension must belong to $I$.

We can now prove Proposition~\ref{Summarized-No_hyperbolic-leaf}.

\begin{proof}[Proof of Proposition~\ref{Summarized-No_hyperbolic-leaf}]
The proof amounts to showing that, within the set of leaves meromorphically parameterized by $\C$ via $X$,
there still exists an invariant subset with Hausdorff dimension strictly greater than $2n-2$ which consists of leaves $L$
satisfying the following condition: the meromorphic map $\phi : \C \rightarrow L$ extends to a {\it holomorphic map} from $\C$
to the {\it enlarged leaf}\, containing $L$.

Let $L$ be a leaf of $\fol$ meromorphically parameterized by the solution $\phi$ of $X$ defined on $\C$ and denote by $\mathcal{D} \subset \C$
the polar set of $\phi$. As previously seen, we can assume that $\mathcal{D}$ is not empty. Fix then $t_0 \in \mathcal{D} \subset \C$
and consider a disc $B_{\varepsilon} (t_0) \subset \C$ around $t_0$ whose radius $\varepsilon$ is small enough to ensure that
$B_{\varepsilon} (t_0) \cap \mathcal{D} = \{ t_0\}$. Next, consider the graph $\overline{{\rm Graph}\, (\phi)} \subset
B_{\varepsilon} (t_0) \times M$ of the meromorphic map $\phi$ restricted to $B_{\varepsilon} (t_0)$. Clearly $\phi$ has no indeterminacy points
points since $B_{\varepsilon} (t_0)$ has dimension~$1$ and a set of indeterminacy points would have codimension at least~$2$.
Therefore, in view of the preceding, the set of points in $\overline{{\rm Graph}\, (\phi)}$ projecting at $t_0$ is finite.
However, since this projection is injective for points lying over $B_{\varepsilon} (t_0) \setminus \{ t_0\}$, it follows
that there is a single point $P = (P_1, P_2)$ in $\overline{{\rm Graph}\, (\phi)}$ projecting to $t_0$. Thus, 
$\overline{{\rm Graph}\, (\phi)} \subset B_{\varepsilon} (t_0) \times M$ is a local analytic curve. Now, since
$\overline{{\rm Graph}\, (\phi)}$ clearly admits an irreducible Puiseux parameterization of the form
$t \mapsto (t, \phi_1(t), \ldots, \phi_n (t))$, it follows that the projection
on $M$ of $\overline{{\rm Graph}\, (\phi)}$ defines again a germ of analytic curve passing through $P_2$ and
contained in the leaf $L$ (except for the point $P_2$ itself). In other words, either $P_2$ is a singular point for $\fol$ and 
(the local branch of) $L$ defines a separatrix for $\fol$ at $P_2$ 
or $P_2$ is regular for $\fol$. In the latter case, it is clear that $P_2$ belongs to $L$ so that $\phi$ extends to $t_0$ as a
holomorphic map with values in $L$ (and hence in the enlarged leaf containing $L$). Thus, we can assume without loss of
generality that the former case happens.

Summarizing what precedes, we can assume the existence of an invariant set having Hausdorff dimension strictly larger than
$2n-2$ and such that all leaves contained in this set define separatrices for $\fol$ at suitable points. Now,
Proposition~\ref{addingendstoleaves} ensures that we can extract from the previous set another invariant set, still having Hausdorff dimension
strictly greater than $2n-2$, whose leaves $L$ are such that the corresponding enlarged leaves $\widehat{L}$ contain all singular points at which
$L$ defines separatrices.
Therefore, for every enlarged leaf $\widehat{L}$ contained in the set in question, $\phi$ extends to a holomorphic function from $\C$ with values
in $\widehat{L}$. This establishes both Proposition~\ref{Summarized-No_hyperbolic-leaf} and
Proposition~\ref{No_hyperbolic-leaf}.
\end{proof}

We are now able to derive Theorem~\ref{classification-foliation} from McQuillan's theorem \cite{mcquillan}
as formulated by Brunella in \cite{IMPA-book}, i.e., as a theorem about 
foliations tangent to a Zariski-dense entire map.

\begin{proof}[Proof of Theorem~\ref{classification-foliation}]
Owing to Proposition~\ref{No_hyperbolic-leaf}, we know that no enlarged leaf of the foliation $\fol$ is a hyperbolic Riemann surface.
Assume now that $N$ is a $2$-dimensional irreducible algebraic set invariant by $\fol$ and not contained in the singular
set of $\fol$. Denote by $\fol_N$ the foliation induced on $N$ by $\fol$. Here the reader is reminded that the leaves of $\fol_N$
might differ from the leaves of $\fol$ contained in $N$, c.f., Definition~\ref{inducedfoliation_definition}.
In the same order of ideas, enlarged leaves of $\fol$ contained in $N$ might differ from enlarged leaves of $\fol_N$. A word is
hence needed in order to conclude that {\it no enlarged leaf of $\fol_N$ is a hyperbolic Riemann surface}.

Consider then a point $p \in N \setminus {\rm Sing}\, (\fol)$ so that
the leaf of $L$ of $\fol$ through~$p$ and the leaf $L^{(N)}$ of $\fol_N$ through~$p$ locally coincide. The corresponding
enlarged leaves will be denoted by $\widehat{L}$ and $\widehat{L}^{(N)}$. To conclude that none of the enlarged leaves
$\widehat{L}^{(N)}$ is a hyperbolic Riemann surface it suffices to show that $\widehat{L}^{(N)}$ is obtained from
$\widehat{L}$ by adding points. In fact, since $\widehat{L}$ is non-hyperbolic, there are non-constant holomorphic
maps from $\C$ to $\widehat{L}$. However, since $\widehat{L}^{(N)}$ is obtained from $\widehat{L}$ by adding points, 
$\widehat{L}$ is naturally identified with a subset of $\widehat{L}^{(N)}$. Thus there are also non-constant holomorphic
maps from $\C$ to $\widehat{L}^{(N)}$ so that $\widehat{L}^{(N)}$ cannot be hyperbolic.

To check that $\widehat{L}^{(N)}$ is obtained from $\widehat{L}$ by adding points, we first consider the difference
between the leaves $L$ and $L^{(N)}$. For this, consider
a curve of singular points contained in $N \cap {\rm Sing}\, (\fol)$. The difference between $L$ and $L^{(N)}$ lies in the fact
that generic points in $1$-dimensional components of $N \cap {\rm Sing}\, (\fol)$ belong to $L^{(N)}$ though they also lie in
${\rm Sing}\, (\fol)$ and hence are certain to not belong to $L$ (whether or not they belong $\widehat{L}$). At this level
it is clear that $L^{(N)}$ is obtained from $L$ by adding points which, in particular, also belong to $\widehat{L}^{(N)}$.
Thus, it only remains to discuss the case of points $q \in N \cap {\rm Sing}\, (\fol)$ {\it that are actually singular for $\fol_N$}.
We need to show that if one such point $q$ is a vanishing end for $L$ it will necessarily be
a vanishing end for $L^{(N)}$ as well. This is, however, an immediate
consequence of the definition of vanishing end since the notion is clearly stable by restriction to subsets, see Section~\ref{Preliminaries}.
It follows that the enlarged leaves of $\fol_N$ are never hyperbolic Riemann surfaces.

In the sequel, we shall exclusively work with the foliation $\fol_N$ so that we can drop subscripts and denote a leaf
of $\fol_N$ simply by $L$ (and the corresponding enlarged leaf by $\widehat{L}$). Consider then an enlarged leaf
$\widehat{L}$ of $\fol_N$. Since $N$ is irreducible and of complex dimension~$2$,
either $\widehat{L}$ is Zariski-dense in $N$ or $\widehat{L}$ is contained in a compact curve (necessarily invariant
under $\fol_N$). In the latter case, note that the compact curve in question is either rational or elliptic
since $\widehat{L}$ is not hyperbolic, though this is not important in what follows. In fact, a well-known theorem
due to Jouanolou \cite{jouanolou} asserts that $\fol_N$ must admit a rational first integral provided that it leaves
invariant infinitely many (irreducible, compact) curves. In turn, if $\fol_N$ admits a rational first integral, then
it is clearly equivalent to a fibration so that the statement of Theorem~\ref{classification-foliation} holds.

It remains to consider the case where $\widehat{L}$ is Zariski-dense. Being non-hyperbolic, there follows that $\widehat{L}$
is a quotient of $\C$. Hence $\fol_N$ is a foliation
that is tangent to an {\it Zariski-dense entire curve}. Owing to McQuillan's theorem as presented in Chapter~9 of \cite{IMPA-book},
one of the following must hold, c.f., Theorem~4 in page 131 of \cite{IMPA-book} as well as Theorem~1 (page 118) and Corollary~1 (page 127)
for the classification of foliations having Kodaira dimension~$0$ or~$1$:
\begin{itemize}
	\item[(1)] Up to a finite ramified covering $\widetilde{N}$ of $N$, the foliation $\fol_N$ is defined by a (global) holomorphic vector field.
	\item[(2)] Up to birational equivalence $\fol_N$ is a Riccati foliation or a turbulent foliation.
	\item[(3)] A fibration.
\end{itemize}
As far as the statement of Theorem~\ref{classification-foliation} is concerned, we can assume without loss of generality
that~(1) holds. Complex compact surfaces equipped with holomorphic vector fields have been studied since long and there are
several ways to rely on the available information to derive Theorem~\ref{classification-foliation}. The argument provided below
has the advantage of relying exclusively on very classical material in the realm of complex (algebraic) geometry.

Consider then the pair $\widetilde{N}$ equipped with a (non-trivial) holomorphic vector field $Z$ whose underlying foliation
is denoted by $\fol_Z$. Here, $\widetilde{N}$ is a possibly singular projective surface. We need to show that up to birational
equivalence, the foliation $\fol_Z$ either is a fibration or as in item~(2) above.

\noindent {\it Claim}. Without loss of generality, $\widetilde{N}$ can be assumed to be smooth.

\noindent {\it Proof of the Claim}. Let ${\rm Sing}\, (\widetilde{N})$ denote the singular locus of $\widetilde{N}$. If not empty,
the set ${\rm Sing}\, (\widetilde{N})$ is constituted by isolated points and irreducible curves. However, ${\rm Sing}\, (\widetilde{N})$
is clearly invariant under the flow of $Z$ so that we actually have:
\begin{itemize}
	\item Every isolated point of ${\rm Sing}\, (\widetilde{N})$ is invariant under $Z$ (i.e., it is a singular point of $Z$).
	
	\item Every irreducible curve $C$ contained in ${\rm Sing}\, (\widetilde{N})$ is globally invariant under $Z$. Note that if the
	curve $C$ is singular, then its singular points are invariant by $Z$ as well.
\end{itemize}

Now, we apply a resolution procedure to $\widetilde{N}$. This procedure consists of two types of operations, namely, normalizations
and blow-ups (see for example \cite{barth}). A normalization must be centered at a smooth curve contained in ${\rm Sing}\, (\widetilde{N})$ and,
since the curve in question is globally invariant by $Z$, the corresponding transform of $Z$ still is a holomorphic vector field.
In turn, blow-ups are centered at isolated points of ${\rm Sing}\, (\widetilde{N})$ or at singular points of a curve in ${\rm Sing}\, (\widetilde{N})$.
In either case, they are invariant under $Z$ so that again the corresponding transform of $Z$ is a holomorphic vector field.
Proceeding by induction, it becomes clear that the singularities of $\widetilde{N}$ can be resolved in such way that the
final transform of the vector field $Z$ is still a holomorphic vector field. This establishes the claim.\qed

In the sequel we assume $\widetilde{N}$ to be smooth. We can also assume that $\widetilde{N}$ does not contain $(-1)$-rational curves, i.e.,
$\widetilde{N}$ can be assumed to be minimal. In fact, the transform of a holomorphic vector field arising from collapsing a
$(-1)$-rational curve is always holomorphic since it is holomorphic away from the point to which the curve collapsed. Since
we are dealing with a surface, the vector field must then extend holomorphically to the point in question since it has codimension~$2$.

The list of compact complex surfaces carrying a {\it non-singular}\,
holomorphic vector field was provided by Mizuhara \cite{Holvf} and it reads:
\begin{enumerate}
	\item A complex torus $\C^2 / \Lambda$;
	
	\item A flat holomorphic fiber bundle over an elliptic curve;
	
	\item An elliptic surface without singular fibers or with singular fibers of type
	$mI_0$ only. In other words, the singular fibers are elliptic curves with finite multiplicity;
	
	\item A Hopf surface or a positive Inoue surface.
\end{enumerate}
In turn, a compact complex surface carrying a holomorphic vector field exhibiting at least one singular point must either be
rational or a surface of class~VII as shown in \cite{car}. Note that Hopf surfaces and Inoue surfaces are themselves examples
of class~VII surfaces, see \cite{barth}. Since $\widetilde{N}$ is projective, and hence algebraic,
surfaces of class~VII can be ruled out of our discussion since they are never
algebraic, see \cite{barth} page~244. Also, a minimal rational surface is either a Hirzebruch
surface $F_n$, $n \neq 1$ or $\CP^2$.

It is immediate to check that holomorphic vector fields on a complex torus are constant and hence satisfy the conditions of
Theorem~\ref{classification-foliation} provided that the torus in question is algebraic. Similarly, holomorphic vector fields
on $\CP^2$ are ``linear'' and it is also immediate to check that they satisfy the conditions of the theorem in question.

In all the remaining cases, the surface $\widetilde{N}$ carries a (possibly singular) fibration of genus~$0$ or~$1$. Since
$Z$ is holomorphic, the classical Blanchard Lemma applies to say that $Z$ must preserve the fibration in question, c.f.,
\cite{Akh}, page~44.
If the fibers
are individually preserved by $Z$, then $Z$ is tangent to the mentioned fibration so that Theorem~\ref{classification-foliation} holds.
Finally, if the flow of $Z$ permutes the fibers of the mentioned fibration, then it is clear that the foliation
$\fol_Z$ will either be Riccati or turbulent depending only on the genus of the fibration. The proof of
Theorem~\ref{classification-foliation} is complete.
\end{proof}

\section{Proofs for Theorems~A and~B}\label{proofsoftheorems}

Deriving Theorems~A and~B from Theorem~\ref{classification-foliation} combined with the resolution theorems in
\cite{danielMcquillan} and \cite{helenajulio_RMS} is rather straightforward. We provide the details below for the convenience of
the reader.

Let us begin by recalling the definition of {\it Liouvillian function}.
Roughly speaking, a Liouvillian function is a function that can be built up from rational functions by means of algebraic
operations, exponentiation, and integrals. The following definition makes this notion accurate (see for example \cite{singer}).

\begin{defi}
	Let $k \subset K$ be differential fields. The field $K$ is said to be a Liouvillian extension of $k$ if there
	is a tower $k = K_0 \subset \cdots \subset K_s =K$ of differential fields such that $K_i = K_{i-1} (t_i)$, for
	$i=1, \ldots , s$, where at least one of the following holds:
	\begin{itemize}
		\item[(1)] The derivative $t_i'$ of $t_i$ lies in $K_{i-1}$ (i.e. $t_i$ is the integral of an element in $K_{i-1}$).
		
		\item[(2)] $t_i \neq 0$ and $t_i'/t_i$ lies in $K_{i-1}$ (i.e. $t_i$ is the exponential of an element in $K_{i-1}$).
		
		\item[(3)] $t_i$ is algebraic over $K_{i-1}$
	\end{itemize}
\end{defi}
If $N$ is a projective manifold, a {\it Liouvillian function}\, on $N$ is therefore an element of a Liouvillian extension
of the field of rational functions on $N$. If, in addition, $N$ is equipped with a foliation $\fol$, then a
{\it Liouvillian first integral}\, for $\fol$ is nothing but a Liouvillian function that is constant
over the leaves of $\fol$.

An immediate consequence of the above definition is that the existence of a Liouvillian first integral $\fol$ is a birational
invariant of the pair $(N, \fol)$. Similarly, a birational equivalence between two manifolds gives rise to a one-to-one correspondence
between the corresponding algebras of Liouvillian functions over $\C$.

\begin{remark}\label{Singerthm}
M. Singer has proved in \cite{singer} that a differential equation with Liouvillian first integrals must admit a transversely
affine structure. This statement was later generalized to codimension~$1$ foliations and an interesting treatment of the topic
can be found in \cite{zoladek}. These statements provide an alternative dynamical point of view in the notion of Liouvillian first integrals.
Note that the converse statement also holds. Namely, if
a codimension~$1$ foliation admits a transverse affine structure, then the developing map (see \cite{standardfoliation})
associated with the affine structure in question provides a Liouvillian first integral for the foliation.
\end{remark}

In the sequel, we place ourselves in the setting of Theorem~A. Thus $M$ is a projective manifold of dimension~$n$ and $X$,
$\fol$, and $\mathcal{M} \subseteq M$ are as in the statement of this theorem. Recall that $\mathcal{M}$ is assumed to have
Hausdorff dimension strictly greater than $2n-2$.

As follows from Proposition~\ref{No_hyperbolic-leaf}, all enlarged leaves of $\fol$ are
either $\Cpp$ or a parabolic Riemann surface so that the first assertion in Theorem~A is already proved. There remains
to prove the second assertion. Let then $N$
be an irreducible (possibly singular)
projective surface left invariant by $\fol$ and not contained in the singular set of $\fol$.
Denote by $\fol_N$ the foliation induced by $\fol$ on $N$.
We can assume from now on that $\fol_N$ has no meromorphic first integral, otherwise there is nothing to prove.
Owing to Jouanolou's theorem \cite{jouanolou} this ensures that only finitely many leaves of $\fol_N$ fail to be Zariski-dense.

In view of Theorem~\ref{classification-foliation}, there follows that $\fol_N$ is either a Riccati foliation or a turbulent one
(up to birational equivalence). A slightly more accurate statement is, however, possible:

\begin{lemma}
\label{amenable_holonomy}
If $\fol_N$ is a Riccati foliation, then its holonomy group is virtually solvable (i.e., it contains a solvable subgroup of finite index).
Similarly, if  $\fol_N$ is a turbulent foliation,
then its holonomy group is solvable and virtually abelian (i.e., it contains an abelian subgroup of finite index).
\end{lemma}

\begin{proof}
Let us first assume that $\fol_N$ is a turbulent foliation and denote by $\calp : N \rightarrow S$ the corresponding genus~$1$ fibration.
Note that the regular fibers of $\calp$ are pairwise equivalent as elliptic curves since local parallel transport along
the leaves of $\fol_N$ yields holomorphic diffeomorphisms between them. If $E$ denote the ``typical'' fiber of $\calp$,
then the holonomy group of $\fol$ is the image of the holonomy representation of the fundamental group of
$S$ in the group of automorphisms ${\rm Aut}\, (E)$ of the elliptic curve $E = \C / \Lambda$, where $\Lambda$ stands for a lattice of $\C$.
The possibilities for the automorphism group ${\rm Aut}\, (E)$ are well-known. First, ${\rm Aut}\, (E)$ is naturally realized
as a subgroup of the affine group of $\C$. In particular, ${\rm Aut}\, (E)$ must be solvable. To see that ${\rm Aut}\, (E)$
is a finite extension of an abelian group, let ${\rm Aut}_0 (E)$ stand for the subgroup of ${\rm Aut}\, (E)$
consisting of those automorphisms whose action on $H^1 (E)$ is trivial. It is well known that ${\rm Aut}_0 (E)$ coincides with the
group of automorphisms induced by translations of $\C$. In particular, ${\rm Aut}_0 (E)$ is abelian. It is also a standard fact that
the index of ${\rm Aut}_0 (E)$
in ${\rm Aut}\, (E)$ can take only on the following values: $2$, $4$, and $6$ (cf. \cite{hartshorne}, page 321). Therefore, 
${\rm Aut}\, (E)$ is a finite extension of an abelian group.

The case of Riccati foliations is slightly different in that the monodromy group is a subgroup of the automorphism
group of the Riemann sphere, namely ${\rm PSL} \, (2,\C)$. In general the holonomy group of a Riccati equation is
far from being virtually solvable.
However, the enlarged leaves of $\fol_N$ are parabolic Riemann surfaces and this imposes strong constraints on
the holonomy group in question. In fact, a subgroup of ${\rm PSL} \, (2,\C)$ that {\it is not}\, virtually solvable contains a free
subgroup on two generators (Tits alternative) and therefore has exponential growth: this forces the Zariski-dense
leaves of $\fol_E$ to have exponential growth as well. Clearly the exponential
growth of the leaves means that they are hyperbolic Riemann surfaces. Since the passage of a leaf to the corresponding
enlarged leaf only adds a discrete set of points to the former, the enlarged leaves must be hyperbolic themselves and
this yields a contradiction with Proposition~\ref{No_hyperbolic-leaf}. The proof of the lemma is complete.
\end{proof}

We are now ready to prove Theorem~A.

\begin{proof}[Proof of Theorem~A]
We keep the preceding notation. In particular, we can assume that the foliation $\fol_N$ is transverse to a (singular)
fibration $\calp : N \rightarrow S$ of genus~$0$ or~$1$. Let us first consider the case in which this fibration has genus~$1$.
Denote by $\Gamma$ the corresponding holonomy group. Lemma~\ref{amenable_holonomy}
establishes that $\Gamma$ is solvable and this yields the desired extension implied in the definition of Liouvillian functions.
In more precise terms, $\Gamma$ is a finite extension of the abelian group $\Gamma_0 = \Gamma \cap {\rm Aut}_0 (E)$: the solutions
are then expressed as a finite extension
of elliptic functions. Alternatively, to make the connection with Singer's theorem in Remark~\ref{Singerthm}, just observe that
the natural ``affine structure'' inherited by $E = \C / \Lambda$ from $\C$ is naturally preserved by $\Gamma$ since this
group is naturally contained in the affine group of $\C$.

Let us now consider the case where $\fol_N$ is a Riccati foliation. Again let $\Gamma$ denote the corresponding holonomy
group which is naturally a finitely generated subgroup of ${\rm PSL} \, (2,\C)$. We can assume $\Gamma$ is infinite, otherwise
all leaves of $\fol_N$ are compact and hence we have a rational first integral. In view of the discussion in
Remark~\ref{Singerthm}, it suffices to prove the existence of a transverse affine structure invariant by $\Gamma$. In particular,
it is enough to check that $\Gamma$ is conjugate to a
subgroup of the affine group ${\rm Aff}\, (\C) \subset {\rm PSL} \, (2,\C)$.

Assume first that $\Gamma$ is discrete and hence a Kleinian group. The group $\Gamma$ is, however, virtually solvable
(Lemma~\ref{amenable_holonomy})
and it is well known that infinite (virtually) solvable Kleinian
group are all {\it elementary} in the sense that their Limit sets consist of one or two points, see for example
\cite{apanasov}, pages~52, 53. The Limit
set being naturally invariant under $\Gamma$, either $\Gamma$ has a fixed point or $\Gamma$ has a subgroup
$\Gamma_2$ of index~$2$ having a fixed point. In the first case, the fixed point of $\Gamma$ can be used to construct
a conjugation between $\Gamma$ and some subgroup of ${\rm Aff}\, (\C)$. In the second case, $\Gamma_2$ has a Liouvillian
first integral and so has $\Gamma$ as a degree~$2$ extension of $\Gamma_2$.

Finally, assume now that $\Gamma$ is not discrete so that its closure $\overline{\Gamma} \subset {\rm PSL} \, (2,\C)$ is a Lie
subgroup of ${\rm PSL} \, (2,\C)$ with non-trivial Lie algebra. Clearly this Lie algebra must be strictly contained in the
Lie algebra of ${\rm PSL} \, (2,\C)$, otherwise $\Gamma$ would be dense in ${\rm PSL} \, (2,\C)$ and thus would fail to be virtually solvable.
This Lie algebra is therefore either abelian or conjugate to an affine Lie algebra. Again this implies that the
group has a fixed point or, at least, contains an index~$2$ subgroup possessing a fixed point (see for example \cite{Ford}, Chapter~6).
The existence of the desired Liouvillian first integral then follows as above. The proof of Theorem~A is complete.
\end{proof}

We can now prove Theorem~B as well. Note that in this statement all solutions of $X$ are assumed to be meromorphic maps
defined on $\C$. Note also that the fundamental theorem on reduction of singular points for one-dimensional foliation on $(\C^3,0)$
first appeared in \cite{danielMcquillan} although it relies heavily on the previous work of Panazzolo \cite{daniel2}. A different
proof is provided in \cite{helenajulio_RMS} along with some sharper results valid for foliations associated with complete vector fields.

\begin{proof}[Proof of Theorem~B]
First assume there is some sequence of blow ups for $M,X$ leading to a vector field $X'$ which happens to be {\it holomorphic}\, in
the corresponding projective manifold $M'$. In this case, a result in  \cite{helenajulio_RMS} asserts that the singularities
of the foliation associated with $X'$ can be reduced by standard blow ups while preserving the holomorphic nature of $X$. Item~(1)
of Theorem~B then holds and the resulting ambient manifold is smooth.

If there is no birational model for $M,X$ where $X$ becomes holomorphic, then we proceed to reduce the singular points of $\fol$
by resorting to either McQuillan-Panazzolo theorem in \cite{danielMcquillan} or to the reduction theorem in \cite{helenajulio_RMS}.
The transformed vector field $X'$ is strictly meromorphic by assumption. Therefore the polar divisor $(X')_{\infty}$ of $X'$ is
invariant by $\fol'$ thanks to Lemma~\ref{lemma_invaraincepoledivisor}. Note that, in applying Lemma~\ref{lemma_invaraincepoledivisor},
we use the assumption that all solutions of $X$ are meromorphic so as to ensure that some of them must ``cross'' the divisor if it were not invariant
(alternatively we might impose some condition ensuring that $(X')_{\infty}$ is ample). To conclude that Item~(2) holds
in this case, we only need to apply Theorem~A to each irreducible component of $(X')_{\infty}$. Theorem~B is proved.
\end{proof}

\section{Final comments and the variation of the leafwise polar set}\label{Appendix}

Let us close this paper with some comments/remarks implicitly appearing in the course of the discussion.

Let us begin by pointing out a difficulty that might go unnoticed when dealing with enlarged leaves as done in
Section~\ref{meromorphicsolutions_brunellatheorem}. At least if all the leaves of $X$ are meromorphically parameterized
by $\C$ via $X$, as in the case of Theorem~B, it is tempting to try to prove Proposition~\ref{Summarized-No_hyperbolic-leaf}
by proceeding as follows. Clearly it basically amounts to constructing the foliated meromorphic immersion required for the definition
of vanishing end of a leaf (at a given singular point). In this regard, we fix a local branch of a
leaf $L$ of $\fol$ defining a separatrix at a point $p \in {\rm Sing}\, (\fol)$.

We need to construct the foliated meromorphic immersion $f: \D^{n-1} \times \D \rightarrow M$
to conclude that $p$ lies in the enlarged leaf arising from $L$. We can assume that the holonomy of $L$ is trivial.
Up to a re-scaling, $f$ is defined on $\{ 0 \} \times \D$ so as to coincide with a solution $\phi$ of $X$ parameterizing $L$
restricted to a suitable open set. Also, we identify $\D^{n-1}$ with a local transverse section $\Sigma$ to $\fol$
so that the extension of $f$ to points of the form $q \times \C$, $q$ different from the origin, would be defined by following the
(suitably normalized) solutions of $X$ through these points.

The difficulty of this approach lies in the fact that it is not easy to conclude that $f$ is a meromorphic map from the
fact that its restriction to each leaf of $\fol$ is meromorphic. For example, there is no a priori reason to conclude that the
closure of the graph of $f$ on $(\D^{n-1} \times \D) \times M$ is an analytic set. This type of conclusion is more accessible
when we have some control on the volume of the images of the leaves or some information on how the polar set of the one-variable
meromorphic maps arising from the solutions of $X$ varies with the leaves (see \cite{bishop} and page 243 of \cite{chirka} for related material).

To overcome this difficulty, the discussion in Section~\ref{Dicriticalnessoffoliations} was carried out and
Proposition~\ref{addingendstoleaves} stated. In the course of the proof of this proposition, or more precisely of Proposition~\ref{dicritical_divisors},
it was shown that every set of separatrices possessing Hausdorff dimension strictly greater than $2n-2$ gives rise a non-invariant
divisor obtained by a suitable sequence of blow ups. With the notation of Section~\ref{Dicriticalnessoffoliations}, let us call a separatrix $\mathcal{C}$
of $\fol$ {\it weakly regular}\, if there exists a finite sequence of blow ups as in~(\ref{sequenceofblowups})
such that the transform $\mathcal{C}_k$ of $\mathcal{C}$ intersects the divisor
$\Pi_{(k)}^{-1} (Z_{k-1})$ transversely at a regular point of $\fol_k$ (in particular $\Pi_{(k)}^{-1} (Z_{k-1})$ is not invariant
by $\fol_k$).

Note that a consequence of the above definition is that the subset of $M$ consisting of all weakly regular separatrices of $\fol$
is {\it open} (possibly empty).
By elaborating slightly further on the argument used in Section~\ref{Dicriticalnessoffoliations},
the following can be proved:

\begin{prop}\label{presenceofweaklyregularseparatrices}
Let $U \subset M$ be a neighborhood of ${\rm Sing}\, (\fol)$ and assume that the set ${\rm Sep}$ of separatrices of $\fol$ in $U$ has
Hausdorff dimension $\kappa > 2n-2$. Then the set ${\rm Sep}_{\rm wreg}$ consisting of all weakly regular separatrices of $\fol$ satisfy
the conditions below:
\begin{itemize}
	\item ${\rm Sep}_{\rm wreg}$ is not empty and open in $M$.
	
	\item The Hausdorff dimension of the set ${\rm Sep} \setminus {\rm Sep}_{\rm wreg}$ is at most $2n-2$.
\end{itemize}
\mbox{} \qed
\end{prop}

A consequence of Proposition~\ref{presenceofweaklyregularseparatrices} is that, for vector fields $X$ as in Theorem~B, the polar
set of the integral curves of $X$ locally varies holomorphically with the leaves away from a set having Hausdorff dimension at most~$2n-2$.

\end{document}